\documentclass[twocolumn]{svjour3}          
\smartqed  
\usepackage{amsmath,amssymb}
\usepackage{a4wide}
\usepackage{graphicx}
\usepackage{subfigure}
\usepackage{amscd}
\usepackage{dsfont}
\usepackage{amscd}
\usepackage{algorithm,algorithmic}
\usepackage{array}

\usepackage{color}
\usepackage[latin1]{inputenc}
\usepackage{wrapfig}

\DeclareMathOperator*{\Argmax}{argmax}

\newcommand{\V}{\mathbb{V}}
\newcommand{\R}{\mathbb{R}}
\newcommand{\Rset}[1]{\mbox{$\mathbb{R}^{#1}$}}
\newcommand{\Z}{\mathbb{Z}}
\newcommand{\N}{\mathbb{N}}

\newcommand{\Sr}{\mathcal{S}}

\newcommand{\te}{\theta}
\newcommand{\hte}{\hat\theta}

\newcommand{\eps}{\varepsilon}

\newcommand{\Te}{\Theta}

\newcommand{\x}{r}

\newcommand{\tS}{S}
\newcommand{\ass}{s}

\newcommand{\bass}{\bar s}

\newtheorem{prop}{Proposition}
\newtheorem{rem}{Remark}

\newcommand{\Kp}{\mathbf{K_p}}

\newcommand{\Kg}{\mathbf{K_g}}

\newcommand{\Syme}{\text{Sym}_{dk_g,*}^+(\mathbb{R})}

\newcommand{\bdy}{\bold{\dy}}
\newcommand{\bdz}{\bold{z}}

\def\pixel{u}
\def\locpixel{v}
\newcommand{\dy}{y}

\newcommand{\pscal}[2] { \left< #1 , #2 \right> }

\newcommand{\lstat}{\textbf{$\pi$}}
\newcommand{\ntrans}{\textbf{$\Pi$}}
\newcommand{\cand}{q_s}
\newcommand{\candMALA}{q_{MALA}}

\newcommand{\smallset}{\texttt{C}}

\def\Gammap{\Sigma_p}
\def\Gammag{\Sigma_g}
\def\I0{I_\alpha}

\def\pk{\x_{p,j}}

\def\gk{\x_{g,j}}

\def\tt{\tau_m}

\def\Kapa{\mathcal{K}}

\def\Xspace{\mathcal{X}}

\newcommand{\obs}{y}
\newcommand{\hid}{z}
\newcommand{\dobs}{q}
\newcommand{\dhid}{l}
\newcommand{\dte}{p}
\newcommand{\dss}{m}

\newcommand{\bhid}{\bar z}

\begin{document}

\title{Convergent Stochastic Expectation Maximization algorithm with efficient sampling in high dimension. Application to deformable template model estimation
}

\titlerunning{Efficient stochastic EM in high dimension}        

\author{St\'ephanie Allassonni\`ere         \and
        Estelle Kuhn 
}


\institute{S. Allassonni\`ere \at
              CMAP Ecole Polytechnique \\
			  Route de Saclay\\
			  91128 Palaiseau, FRANCE
              Tel.: +331.69.33.45.65\\
              \email{stephanie.allassonniere@polytechnique.edu}           
           \and
           E. Kuhn \at
              INRA\\
			  Domaine de Vilvert\\
			  78352 Jouy-en-Josas, FRANCE\\
}

\date{Received: date / Accepted: date}

\maketitle

\begin{abstract}
	Estimation in the deformable template model is a big challenge in image analysis. The issue is to estimate an atlas of a population. This atlas contains a template and the corresponding geometrical variability of the observed shapes. The goal is to propose an accurate algorithm with low computational cost and with theoretical guaranties of relevance. This becomes very demanding when dealing with high dimensional data which is particularly the case of medical images. We propose to use an optimized Monte Carlo Markov Chain method into a stochastic Expectation Maximization  algorithm in order to estimate the model parameters by maximizing the likelihood. In this paper, we present a new Anisotropic Metropolis Adjusted Langevin Algorithm which we use as transition in the MCMC method. We first prove that this new sampler leads to a geometrically uniformly ergodic Markov chain.
	We prove also that under mild conditions, the estimated parameters converge almost surely and are asymptotically Gaussian distributed. The methodology developed is then tested on handwritten digits and some 2D and 3D medical images for the deformable model estimation. More widely, the proposed algorithm can be used for a large range of models in many fields of applications such as pharmacology or genetic.
\end{abstract}

\section{Introduction}

We consider here the deformable template model introduced for Computational Anatomy in \cite{gm98}. This model, which has demonstrated great impact in image analysis, was developed and analyzed later on by many groups (among other  \cite{mty02,marsland04:_const_diffeom_repres_group_analy,VPPA2009,hippodata}). It offers several major advantages. First, it enables to describe the population of interest by a digital anatomical template. 
It also captures the geometric variability of the population shapes through the modeling of deformations of the template which match it to the observations. Moreover, the metric on the space of deformations is specified in the model as a quantification of the deformation cost. Not only describing the population, this generative model also allows to sample synthetic data using both the template and the geometrical metric of the deformation space which together define the atlas. Nevertheless, the key statistical issue is how to estimate efficiently and accurately these parameters of the model from an observed population of images.

Several numerical methods have been developed mainly for the estimation of the template image (for example \cite{ctcg95,joshi04:_unbias}). Even if these methods lead to visual interesting results on some training samples, they suffer from a lack of theoretical properties raising the question of the relevance of the output and are not robust to noisy data. Another important contribution toward the statistical formulation of the template estimation issue was proposed in \cite{glasbey01}. However interesting this approach is not entirely satisfactory since the deformations are applied to discrete observations requiring some interpolation. Moreover it does not formulate the analysis in terms of a generative model which appears very attractive as mentioned above. 
To overcome these lacks, a coherent statistical generative model was formulated in \cite{AAT}.  For estimating all the model parameters, the template image together with the geometrical metric, the authors proposed a deterministic algorithm based on an approximation of the well-known Expectation Maximization (EM) algorithm  (see \cite{DLR}), where the posterior distribution is replaced by a Dirac measure on its mode (called FAM-EM). 
However, such an approximation leads to the non-convergence of the estimates highlighted when considering noisy observations. 

One solution to face this problem is to consider a convergent stochastic approximation of the EM (SAEM) algorithm which was proposed in \cite{DLM}. An extension using Monte Carlo Markov Chain (MCMC) methods was developed and studied in \cite{kuhnlavielle} and \cite{aktdefmod} allowing for wider applications.
To apply this extension to the deformable template model, the authors in \cite{aktdefmod} chose a Metropolis Hastings within Gibbs sampler (also called hybrid Gibbs) as MCMC method since the variables to sample were of large dimension (the usual Metropolis Hastings algorithm providing low acceptation rates).
This estimation algorithm has been proved convergent and performs very well on very different kind of data as presented in \cite{aktSFDS2010}. 
Nevertheless, the hybrid Gibbs sampler becomes computationally very expensive when sampling very high dimensional variables. Although it reduces the dimension of the sampling to one which enables to stride easier the target density support, it loops over the sampling variable coordinates, which becomes computationally unusable as soon as the dimension is very large or as the acceptation ratio involves heavy computations.
To overcome the problem of computational cost of this estimation algorithm, some authors propose to simplify the statistical model constraining the correlations of the deformations (see \cite{rsc09,mlmd2011}).  

Our purpose in this paper is to propose an efficient and convergent estimation algorithm for the deformable template model in high dimension without any constrains.  With regards to the above considerations,  the computational cost of the estimation algorithm can be reduced by optimizing the sampling scheme in the MCMC method.
\\

The sampling of high dimensional variables is a well-known difficult challenge. 
In particular, many authors have proposed to use the Metropolis Adjusted Langevin Algorithm (MALA) (see \cite{rt96} and \cite{st99_1}). This algorithm is a particular random walk Metropolis Hastings sampler. Starting from the current iterate of the Markov chain, one simulates a candidate with respect to a Gaussian proposal with an expectation equal to the sum of this current iterate and a drift related to the target distribution. The covariance matrix is diagonal and isotropic. This candidate is accepted or rejected with a probability given by the Metropolis Hastings acceptance ratio.

Some modifications have been proposed in particular to optimize the covariance matrix of the proposal in order to better stride the support of the target distribution (see \cite{st99_2,atchade2006,mr2011,gc2011}). 
In \cite{atchade2006} and \cite{mr2011}, the authors proposed to construct adaptive MALA chains for which they prove the geometric ergodicity of the chain uniformly on any compact subset of its parameters. Unfortunately, this technique does not take the whole advantage of changing the proposal using the target distribution. In particular, the covariance matrix of the proposal is given by a stochastic approximation of the empirical covariance matrix. This choice seems completely relevant as soon as the convergence toward the stationary distribution is reached. However, it does not provide a good guess of the variability during the first iterations of the chain since it is still very dependent on the initialization. This leads to chains that may be numerically trapped. Moreover, this particular algorithm may require a lot of tuning parameters. Although the theoretical convergence is proved, this algorithm may be very difficult to optimize in practice \emph{into} an estimation process.

Recently, the authors in \cite{gc2011} proposed the Riemann manifold Langevin algorithm in order to sample from a target density in high dimensional setting with strong correlations. This algorithm is also a MALA based one for which the choice of the proposal covariance is guided by the metric of the underlying Riemann manifold. It requires to evaluate the metric, its inverse as well as its derivatives. The proposed well-suited metric is the Fisher-Rao information matrix or its empirical value. However, in the context we are dealing with, the real metric, namely the metric of the space of non-rigid deformations, is not explicit preventing from any use of it (the simplest case of the 3-landmark-matching problem is calculated in \cite{mmm2012} leading to a very intricate formula which is difficult to extend  to more complex models). Moreover, if we consider the constant curvature simplification suggested in \cite{gc2011}, one still needs to invert the metric which may be neither explicit nor computationally tractable. Note that these constrains are common with other application fields such as genetic or pharmacology, where models are often complex.\\

For all these reasons, we propose to adapt the MALA algorithm in the spirit of both works in \cite{atchade2006} and \cite{gc2011} to get an efficient sampler into the stochastic EM algorithm. 
Therefore, we propose to sample from a proposal distribution which has the same expectation as the MALA but using a full anisotropic covariance matrix based on the anisotropy and correlations of the target distribution. This sampler will be called AMALA in the sequel.
 The expectation is obtained as the sum of the current iterate plus a drift which is proportional to the gradient of the logarithm of the target distribution. We construct the covariance matrix as a regularization of the Gram matrix of this drift.
We prove the geometric ergodicity uniformly on any compact set of the AMALA assuming some regularity conditions on the target distribution. We also prove the almost sure convergence of the parameter estimated sequence generated by the coupling of AMALA and SAEM algorithms (AMALA-SAEM) toward the maximum likelihood estimate under some regularity assumptions on the model. Moreover, we prove a Central Limit Theorem for this sequence under usual conditions on the model.

We test our estimation algorithm on the deformable template model for estimating hand-written digit atlases from the USPS database and medical images of corpus callosum (2D) and of dendrite spine excrescences (3D). The proposed estimation method is compared with the results obtained from the FAM-EM algorithm and from the MCMC-SAEM algorithm using different samplers namely the hybrid Gibbs sampler, the  MALA and the adaptive MALA proposed in \cite{atchade2006} previously introduced. The comparison is also made via classification rates on the USPS database. These experiments demonstrate the good behavior of our method in both the accuracy of the estimation and the low computational cost in high dimension.
\\

The paper is organized as follows. In Section \ref{sec:BME}, we recall the Bayesian Mixed Effect (BME) template model. In Section \ref{sec:AMALASAEM}, we consider the maximum likelihood estimation issue in the general framework of missing data models. We pre- sent our stochastic version of the EM algorithm using the AMALA sampler. The convergence properties are established in Section~\ref{sec:theoretical_properties}.  Section \ref{sec:appli} is devoted to the experiments on the BME template estimation. Finally, we give some conclusion in Section \ref{sec:conclusion}. The proofs are postponed in Section \ref{sec:proofs}.

\section{Description of the Bayesian Mixed Effect (BME) Template model}
\label{sec:BME}

The deformable template model aims at summarizing a population of images by two quantities. The first one is a mean image called template which has to represent a relevant shape as one could find in the population. The second quantity represents the variance in the space of shapes. This corresponds to the geometrical variability around the mean shape. Let us now describe the deformable template model more precisely.

We consider the hierarchical Bayesian framework for dense deformable template developed in \cite{AAT} where each image in a population is assumed to be generated as a noisy and randomly deformed version of the template. 

The database is composed of $n$ grey level images $(\dy_i)_{1\leq i\leq n}$ observed on a grid $\Lambda$ of pixels (or voxels) 
included in a continuous domain $D \subset \R^d$, (typically $D=[-1,1]^{d}$ where $d$ equals $2$ or $3$).
The expected template $I_0 : \R^{d} \to \R$ takes its values in the continuous domain. 
Each observation $\dy$ is assumed to be a discretization on $\Lambda$ of a random deformation of this template plus an independent noise. 
Therefore, there exists an unobserved deformation field (also called mapping) $m:\R^d \to \R^d$ such that for $\pixel \in \Lambda$
\begin{eqnarray*}
\dy(\pixel) = I_0(\locpixel_\pixel - m(\locpixel_\pixel)) + \sigma \epsilon (\pixel)\ ,
\end{eqnarray*}
where $ \sigma \epsilon$ denotes the independent additive noise and $\locpixel_\pixel $ is the location of pixel (or voxel) $\pixel$.

Considering the template and the deformations as continuous functions would lead to a dense problem. 
The dimension is reduced assuming that both elements belong to a subset of fixed Reproducing Kernel Hilbert Spaces (RKHS) $V_p$ and $V_g$ defined by their respective kernels $K_p$  and $K_g$. More precisely,  
let $(\pk)_{1\leq j\leq k_p}$ -respectively $(\gk)_{1\leq j\leq k_g}$- be some fixed control points in the domain $D$: 
there exist $\alpha \in \R^{k_p}$ -resp. $z \in \R^{k_g} \times \R^{k_g}$- such that for all $ \locpixel$ in $D $:
\begin{eqnarray}\label{template}
\I0(\locpixel) = (\Kp \alpha)(\locpixel) = \sum\limits_{j=1}^{k_p} 
K_p (\locpixel,\pk) \alpha^j  \\ m_z(\locpixel) = (\Kg z)(\locpixel) = \sum\limits_{j=1}^{k_g} K_g (\locpixel,\gk)
z^j\,.
\end{eqnarray}

For clarity, we write $\bdy=(\dy_i)_{1\leq i\leq n}$ for the $n-$tuple of observations and $\bdz=(z_i)_{1\leq i\leq n}$ for the $n-$tuple of unobserved variables defining the deformations. The statistical model on the observations is chosen as follows:
\begin{equation}\label{pre-priors}
\left\{
  \begin{array}[h]{l}
  \bdz\sim \otimes_{i=1}^n\mathcal{N}_{dk_g}(0,\Gamma_g)\  |\
\Gamma_g \, ,\\\\
\bdy\sim
  \otimes_{i=1}^n\mathcal{N}_{|\Lambda|} (m_{z_i}\I0,\sigma^2 Id_{|\Lambda|})\ |\
  \bdz,\alpha, \sigma^2 \, ,
  \end{array}\right.
\end{equation}
where $\otimes$ denotes the product of independent variables and $mI_\alpha(\pixel) = I_\alpha(\locpixel_\pixel - m(\locpixel_\pixel))$, for $\pixel$ in $\Lambda$. 
The parameters of interest are the template $\alpha$, the noise variance $\sigma^2$  and the deformation covariance matrix $\Gamma_g$. We assume that $\te=(\alpha, \sigma^2, \Gamma_g)$ belongs to the parameter space $\Te$:
\begin{multline}
  \Theta \triangleq \{\ \theta=(\alpha,\sigma^2,\Gamma_g)\ |\
\alpha\in\mathbb{R}^{k_p},\ |\alpha| <R, \\  \ \sigma>0,\ \Gamma_g \in\Syme\ \}\,,
\end{multline}
where $\Syme$ is the cone of real positive $d k_g \times d k_g$ definite
symmetric matrices, $R$ is an arbitrary positive constant and $d$ is the space dimension (typically $2$ or $3$ for images).

Since we aim at dealing with small size samples and high dimensional parameters, we work in the Bayesian framework and we introduce priors on the parameters. In addition of guiding the estimation it regularizes the estimation as shown in \cite{AAT}. The priors are all independent:
$\te=(\alpha, \sigma^2,\Gamma_g)\sim \nu_p\otimes\nu_g$ where
\begin{equation}\label{priors} \left\{
\begin{array}{l}
\displaystyle{\nu_p(d\alpha, d\sigma^2)} \varpropto 
\displaystyle{\exp \left(-\frac{1}{2} (\alpha-\mu_p)^T (\Gammap)^{-1} (\alpha - \mu_p) \right)}\times \\
\displaystyle{
\left(
\exp\left(-\frac{\sigma_0^2}{2 \sigma^2}\right) \frac{1}{\sqrt{\sigma^2}}
\right)^{a_p}
  d\sigma^2
 d\alpha}, \ a_p \geq 3 \, , \\
\displaystyle{\nu_g(d\Gamma_g) \varpropto } \left( \exp(-\langle
 \Gamma_g^{-1} ,
 \Gammag \rangle_F /2) \frac{1}{\sqrt{|\Gamma_g|}}  \right)^{a_g}
 d\Gamma_g, \\ a_g\geq 4k_g+1
\, .
\end{array} \right.
\end{equation}
For two matrices $A,B$ we define the Frobenius inner product by $\langle A,B \rangle_F \triangleq tr(A^T B)$.\\

 Parameter estimation for this model is then performed by Maximum A Posteriori (MAP)~: 
\begin{eqnarray} \label{MAP}
\hat{\theta} = \Argmax_{\theta \in \Theta} q(\theta|\bdy)\,,
\end{eqnarray}
where $q(\theta|\bdy)$ is the posterior density of $\theta$ conditional on $\bdy$.
The existence and consistency of the MAP estimator for the BME template model has been proved in \cite{AAT}.\\

Note that this model belongs to a more general class called mixed effect models. The fixed effects are the parameters $\theta$ and the random effects are the deformation coefficients $\bdz$. The estimation issue in this class is treated in the same way as the likelihood maximization problem in the more general framework of incomplete-data models. Therefore, the next section will be presented in this general setting in which the proposed algorithm applies.

\section{Maximum likelihood  estimation}
\label{sec:AMALASAEM}

\subsection{Maximum likelihood estimation for incomplete data setting} 
\label{sub:incomplete_data_setting}

We consider in this section the standard incomplete data (or partially-observed-data) setting and recall the usual notation. 
We denote by $\obs\in \Rset{\dobs}$ the observed data and by $\hid\in  \Rset{\dhid}$ the missing 
data, so that we obtain the complete data $(\obs,\hid)\in  \Rset{\dobs+\dhid}$ for some  $\dobs \in  \N^*$ and $\dhid  \in
\N^*$. 
We consider these data as random vectors. Let  $\mu'$  be  a  $\sigma$-finite  measure on  $\Rset{\dobs+\dhid}$  and  $\mu$  the
restriction  of $\mu'$  to  $\Rset{\dhid}$  generated by the projection $(\obs,\hid)  \mapsto \hid$.  
We assume that the probability density function (pdf) of the random vector $(\obs,\hid)$  belongs to 
$\mathcal{P}=\{f(\obs,\hid;\te),\te \in  \Te\}$, a  family of parametric probability
density functions on $\Rset{\dobs+\dhid}$ w.r.t. $\mu'$, where $\Te \subset \Rset{\dte}$.
 Therefore, the observed  likelihood
(i.e. the incomplete-data likelihood) is defined for some $\theta \in \Theta$ by:
\begin{equation}\label{obslik}
g(\obs;\te) \triangleq \int f(\obs,\hid;\te)\mu(d\hid).
\end{equation}

Our purpose is to find  the maximum likelihood estimate that is the value  $\hte_g$ in $\Te$ that maximizes the observed likelihood $g$ given a sample of observations.  However, this maximization can often not be done analytically because of the integration involved in  \eqref{obslik}. A powerful tool which enables to compute this maximization in such a setting is the Expectation Maximization (EM) algorithm (see \cite{DLR}). It is an iterative procedure which consists of two steps. First, the so-called E-step computes the conditional expectation of the complete log-likelihood  using the current parameter value. 
Second, the M-step achieves the update of the parameter by maximizing this expectation over $\Te$.
However, the computation of this expectation is often intractable analytically. Therefore, alternative 
procedures have been proposed. We are particularly interested in the Stochastic Approximation EM (SAEM) algorithm (see \cite{DLM})
because of its  theoretical convergence property and its small computation time.
 In this stochastic algorithm, 
the usual E-step is replaced by two steps, the first one corresponding to the simulation of
realizations of the missing data, the second one to the computation of a stochastic approximation of the complete log-likelihood using these simulated values. It 
can be shown  under weak regularity conditions that the sequence generated by this algorithm converges
almost surely toward a local maximum of the observed likelihood (see \cite{DLM}).

Nevertheless the simulation step requires some attention. In the  SAEM algorithm the 
simulated values of the missing data have to be drawn from the posterior distribution defined by:
\begin{equation*}
p(\hid|\obs;\te)\triangleq 
\left \{ 
\begin{array}{ll} 
f(\obs,\hid;\te) / g(\obs;\te) & \mbox{if } g(\obs;\te) \not = 0 \\
0                           & \mbox{if } g(\obs;\te) = 0  \,.
\end{array} 
\right.
\end{equation*}

When not possible, the extension
using  MCMC method (see \cite{kuhnlavielle,aktdefmod}) allows to apply the SAEM algorithm using simulations obtained from some transition
probability of an ergodic Markov chain having the targeted posterior distribution as stationary distribution.
Methods like Metropolis Hastings algorithm or Gibbs sampler are useful to perform this 
assignment. However, this becomes very challenging in high dimensional setting. Indeed, when the MCMC 
procedure has to explore a space of high dimension, its  convergence may occur in practice only after
a possibly infinite time. Thus, it is necessary to optimize this MCMC procedure. 
This is what we will propose in the following paragraph.

\subsection{Description of the sampling method: Anisotropic Metropolis Adjusted Langevin Algorithm}
\label{sec:AMALA}

We propose an anisotropic version of the well-known Metropolis Adjusted Langevin Algorithm (MALA).
So let us first recall the steps of this algorithm. 
Let $\Xspace$ be an open subset of $\R^\dhid$, the $\dhid-$dimensional Euclidean space equipped with its Borel $\sigma-$algebra $\mathcal{B}$. Let us denote $\lstat$ the pdf of the target distribution with respect to the Lebesgue measure on $\Xspace$. We assume that $\lstat$ is positive continuously differentiable.
At each iteration $k$ of this algorithm, a candidate $X_c$ is simulated with respect to the Gaussian distribution with expectation $X_k + \frac{\sigma^2}{2} D(X_k)$ and covariance $\sigma^2 Id_{\dhid}$ where $X_k$ is the current value, 
\begin{equation}
	\label{eq:D}
D(x) = \frac{b}{\max(b,|\nabla \log \lstat(x)|)}\nabla \log \lstat(x)\,,	
\end{equation}
 $Id_{\dhid}$ is the identity matrix in $\R^{\dhid}$ and $b>0$ is a fixed truncation threshold. Note that the truncation of the drift $D$ was already suggested in \cite{grsBook} to provide more stability. In the following, we denote $\candMALA(x,\cdot)$ the pdf of this Gaussian candidate distribution starting from $x$. Given this candidate, the next value of the Markov chain is updated using an acceptance ratio $\alpha_{MALA}(X_k,X_c)$ as follows:
$X_{k+1} = X_c$ with probability 
\begin{equation}
	\label{eq:acceptrate}
\alpha_{MALA} (X_k,X_c) = \min\left(1,\frac{\lstat(X_c)\candMALA(X_c,X_k)}{\candMALA(X_k,X_c)\lstat(X_k)}\right)
\end{equation}
and $X_{k+1}=X_k$ with probability  $1-\alpha_{MALA}(X_k,X_c)$. This provides a transition kernel $\ntrans_{MALA}$ of this form: for any Borel set $A\in \mathcal{B}$
\begin{multline}
	\label{eq:transitionkernelMALA}
	\ntrans_{MALA}(x,A) = \int_A \alpha_{MALA}(x,z) \candMALA(x,z)dz + \\ \mathds{1}_{A} (x)\int_{\Xspace}(1-\alpha_{MALA}(x,z))\candMALA(x,z)dz\,.
\end{multline}

The Gaussian proposal of the MALA algorithm is optimized with respect to its expectation guided by the Langevin diffusion. One step further is to optimize also its covariance matrix. A first work in this direction was proposed in \cite{atchade2006}. The covariance matrix of the proposal is given by a projection of a stochastic approximation of the empirical covariance matrix. It produces an adaptive Markov chain. This process involves some additional tuning parameters which have to be calibrated. Since our goal is to use this sampler in an estimation algorithm, the sampler has at each iteration a different target distribution (depending on the current estimate of the parameter). Therefore, the optimal tuning parameter may be different along the iterations of the estimation process.  Although we agree with the idea of using  adaptive chain, we prefer taking the advantage of the dynamic of the estimation algorithm.
On the other side, an intrinsic solution has been proposed in \cite{gc2011} where the covariance matrix is given by the metric of the Riemann manifold of the variable to sample. Unfortunately, this metric may not be accessible and its empirical approximation not easy to compute. This is particularly the case in the BME template model. 
\\

For these reasons, we propose a sampler in the spirit of \cite{atchade2006}, \cite{gc2011} or \cite{mr2011} however not providing an adaptive chain as motivated above. The adaption comes from the dependency of the target distribution with respect to the parameters of the model which are updated along the estimation algorithm. The proposal remains a Gaussian distribution but both the drift and the covariance matrix depend on the gradient of the target distribution. At the $k^{th}$ iteration, we are provided with $X_k$. The candidate is sampled from the Gaussian distribution with expectation $X_k+ \delta D(X_k)$ and covariance matrix $\delta \Sigma(X_k)$ denoted in the sequel

\noindent $\mathcal{N}\left(X_k+ \delta D(X_k), \delta \Sigma(X_k)\right)$ where 
$\Sigma(x) $ is given by~: 

\begin{equation}
	\label{eq:covarAMALA}
	\Sigma(x)= \varepsilon Id_{\dhid} + D(x) D(x)^T\,,
\end{equation}
$D$ is defined in Equation~\eqref{eq:D} and  $\varepsilon >0$ is a small regularization parameter. Note that the threshold parameter $b$ leads to a symmetric positive definite covariance matrix with bounded non zero eigenvalues.
We introduce the gradient of $\log \lstat$ into the covariance matrix to provide an anisotropic covariance matrix depending on the amplitude of the drift at the current value. When the drift is large, the candidate is likely to be far from the current value. This large step may not be of the right amplitude and a large variance will enable more flexibility. Moreover, this enables to explore a larger area around these candidates which would not be possible with a fixed variance.  On the other hand, when the drift is small in a particular direction, it means that the current value is within a region of high probability for the next value of the Markov chain. Therefore, the candidate should not move too far neither with a large drift nor with a large variance. This enables to sample a lot around large modes which is of particular interest.
This covariance also enables to treat the directions of interest with different amplitudes of variances as the drift already does. It also provides dependencies between coordinates since the directions of large variances are likely to be different from the Euclidean axis. 
This is taken into account here by introducing the Gram matrix of the drift into the covariance matrix. 

We denote by $q_c$ the pdf of this proposal distribution.
The transition kernel becomes: for any Borel set $A$.

\begin{multline}
	\label{eq:transitionkernel}
	\ntrans(x,A) = \int_A \alpha(x,z) q_c(x,z)dz + \\ \mathds{1}_{A} (x)\int_{\Xspace}(1-\alpha(x,z))q_c(x,z)dz\,,
\end{multline}
where 
\begin{equation}
	\label{eq:acceptationAMALA}
\alpha(X_k,X_c)=\min\left(1,\frac{\lstat(X_c)q_c(X_c,X_k)}{q_c(X_k,X_c)\lstat(X_k)}\right)\,.
\end{equation}

\subsection{Description of the  stochastic estimation algorithm}
\label{subsec:SAEM_AMALA}

Back to the stochastic estimation algorithm, the target distribution of the sampler is the posterior distribution  $ p( \cdot |\obs; \theta)$.

The four steps of the proposed AMALA-SAEM algorithm are detailed in this subsection~: simulation,
 stochastic approximation, truncation on random boundaries
and maximization steps. At each iteration $k$ of the algorithm, simulated values of the 
missing data are drawn from the transition probability of the AMALA algorithm described
 in Section \ref{sec:AMALA} with the current value of the parameters. Then, a stochastic approximation of the complete log-likelihood is computed using these simulated values for the missing data and is truncated using random boundaries. 
Finally, the parameters are updated by maximizing this quantity over $\Te$.

We consider here only parametric models $\mathcal{P}$ which belong to the curved 
exponential family, this means that the complete
likelihood  $f(\obs,\hid;\te)$ can be written as:
\begin{equation*}
\label{expform}
  f(\obs,\hid;\te) = \exp \left[ -\psi(\te) + \langle \tS(\hid),
  \phi(\te)\rangle \right] \, ,
\end{equation*}
where $\pscal{\cdot}{\cdot}$ denotes the Euclidean scalar product,
the sufficient statistics $\tS$ is a  function on $\Rset{\dhid}$,
 taking its
values in a subset $\Sr$ of $\Rset{\dss}$ and $\psi$, $\phi$ are
two  functions on $\Te$ (note that $\tS$, $\phi$ and $\psi$ may
depend also on $\obs$, but  we
omit this dependency for simplicity). This condition is usual in the framework of EM algorithm applications and it is fulfilled by large range of models even complex ones as the BME template model.  
Therefore the stochastic approximation can be done either on the sufficient statistics $\tS$ of the model 
or on the complete log-likelihood using a positive step-size sequence $(\gamma_k)_{k\in\N}$. 

Concerning the truncation procedure, we introduce a sequence of increasing compact 
subsets of $\Sr$ denoted by $(\Kapa_q)_{q\geq 0} $
such that $\cup_{q\geq 0} \Kapa _q = \Sr $
and $ \Kapa _q \subset  \text{int}(\Kapa _{q+1}) ,$ for all $q \geq 0$. Let also
$(\varepsilon_q)_{q\geq 0}  $ be a monotone non-increasing sequence of positive
numbers and $\mathrm{K}$ a compact subset of $\R^{l}$. 
At iteration $k$ we simulate a   value $\bhid$ for the 
missing data from the Anisotropic Metropolis Adjusted Langevin Algorithm using the current 
value of the parameter $\te_{k-1}$. We compute the  associated stochastic approximation of the sufficient 
statistics of the model $\bass$. 
If it does not  wander outside the current
compact set $\Kapa_k$ and if it is not too far from its previous value $\ass_{k-1}$, 
we keep the possible proposed values for $(\hid_k,\ass_k)$.
As soon as one of these conditions is
not fulfilled, 
we reinitialize the sequences of
 $\hid$ and $\ass$ using
a projection (for more details see \cite{andrieumoulinespriouret} ) and  we
increase the size of the compact set used for the truncation. 
As explained in \cite{andrieumoulinespriouret}, the re-projections act as a drift  as they force the chain to come back to a compact set when it grows too rapidly. It reinitializes the algorithm with a smaller step size. However, as the chain has an unbounded support, it requires the use of adaptive truncations. As we shall see in the Proof section (and already noted in \cite{andrieumoulinespriouret}), the limitation imposed on the increments of the sequence is required in order to ensure the convergence of the whole algorithm.

Concerning the maximization step, we denote by $L$ the function defined on $\Sr \times \Te$ taking values in $\Rset{}$
equaled for all $(\ass,\te)$ to $L(\ass,\te)=-\psi(\te) + \langle \ass,  \phi(\te)\rangle$. We assume that there exists 
a function $\hte$ defined on  $\Sr$ taking values in $\Te$ such that
$$\forall \te \in \Te \ \forall \ass \in \Sr \ L(\ass,\hte(\ass))\geq L(\ass,\te).$$ 
Finally we update the parameter using  the value of the function $\hte$ evaluated
in  $\ass_k$.
\\

The complete algorithm is summarized in Algorithm \ref{algo:AMALASAEM}. It only involves three parameters: $b$ the threshold for the gradient which appears in the expectation as well as in the covariance matrix, $\delta$ the scale on this gradient and $\varepsilon$ a small regularization parameter to ensure a positive definite covariance matrix. The scale  $\delta$ can be easily optimized looking at the data we are dealing with to adapt to the range of the drift. The value of the threshold $b$ is in practice never reached. 
The practical choices for the sequences $(\gamma_k)_k$ and $(\varepsilon_k)_k$ of positive step sizes used in the stochastic approximation and the tuning parameters
will be detailed in the section devoted to the experiments.

\begin{algorithm}
 \caption{AMALA within SAEM}
\label{algo:AMALASAEM}
\begin{algorithmic}
\FORALL{$k=1:k_{end}$}
\STATE\texttt{Sample} $\hid_c $ with respect to $\mathcal{N}(\hid_{k-1}+ \delta D(\hid_{k-1},\theta_{k-1}), \delta 
\Sigma(\hid_{k-1},\theta_{k-1}))$ whose pdf is denoted $q_{s_{k-1}}(\hid_{k-1}, . )$
where  
\begin{equation*}
	\left\{
	\begin{array}{ll}
	D(\hid_{k-1},\theta_{k-1}) =&  \frac{b\nabla \log p(\hid_{k-1}|\obs;\te_{k-1})}{\max(b,|\nabla \log p(\hid_{k-1}|\obs;\te_{k-1})|)} 
		   \\
\\
	\Sigma(\hid_{k-1},\theta_{k-1})=&  D(\hid_{k-1},\theta_{k-1})D(\hid_{k-1},\theta_{k-1})^T +
	\\ 
	&\varepsilon Id_{\dhid}.
	\end{array}\right.
\end{equation*}
\vspace{0.3cm}
\STATE\texttt{Compute} the acceptance ratio 
	$\alpha_{\theta_{k-1}}(\hid_{k-1},\hid_c)$ as defined in Eq.~\eqref{eq:acceptationAMALA}. 
\STATE\texttt{Sample}  $ \bhid=\hid_c$ with probability  $\alpha_{\theta_{k-1}}(\hid_{k-1},\hid_c)$ and $\bhid=\hid_{k-1}$ with probability $1-\alpha_{\theta_{k-1}}(\hid_{k-1},\hid_c)$
\STATE\texttt{Do the stochastic approximation} $$\bass = \ass_{k-1} + \gamma_k \left(\tS(\bhid) - \ass_{k-1}\right),$$
where $(\gamma_k)_k$ is a sequence of positive step sizes.
\IF{ $\bass \in \Kapa_{\kappa_{k-1}} $
and $\|\bass - \ass_{k-1}\|\leq \varepsilon_{\zeta_{k-1}}  $}
\STATE Set $( \hid_k,\ass_k) = (\bhid,\bass)$
and
$\kappa_k=\kappa_{k-1} $, $\nu_k=\nu_{k-1}+1$, $\zeta_k=\zeta_{k-1}+1$
\ELSE
\STATE set $( \hid_k,s_k) = (\tilde{\hid},\tilde{\ass}) \in
 \mathrm{K}\times \Kapa_0 $ and  $\kappa_k=\kappa_{k-1}+1 $, $\nu_k=0 $, $\zeta_k =
\zeta_{k-1} + \Psi(\nu_{k-1})$
\STATE where $\Psi : \ \N \to \Z$ is a function such that
$\Psi(k)> -k$ for any $k$
\STATE and
$(\tilde{\hid},\tilde{\ass})$  is chosen arbitrarily.
\ENDIF
\STATE\texttt{Update} the parameter $$\te_{k}=\hte(\ass_k)$$
\ENDFOR
\vspace{0.5cm}
\end{algorithmic}
\end{algorithm}

\section{Theoretical Properties} 
\label{sec:theoretical_properties}

\subsection{Geometric ergodicity of the AMALA}
\label{sec:mod.3}

Let $\Sr$ be a subset of $\Rset{\dss}$ for some positive integer $\dss$.
 Let $\mathcal{X}$ be a measurable subspace of $\Rset{\dhid}$ for some positive integer 
$\dhid$. 
Let $(\lstat_s)_{s \in \Sr}$ be a family of positive continuously differentiable
 probability density functions with respect to the Lebesgue measure on $\mathcal{X}$. 
For any $s\in \Sr$, denote by $\ntrans_s$ the transition kernel corresponding to the AMALA procedure described in Section \ref{sec:AMALA} with stationary distribution $\lstat_s$.
We prove in the following proposition that each kernel of the family $(\ntrans_s)_{s\in\Sr}$ is uniformly geometrically ergodic and that this property holds uniformly in $s$ on any compact subset $\Kapa$ of $\Sr$.

We require a usual assumption  on the stationary distributions namely the so-called super-exponential property given by:

\begin{description}
	\item[\textbf{(B1)}] For all $s\in\Sr$, the density $\lstat_s$ is positive with continuous first derivative such that:
	\begin{equation}
		\label{eq:superexp1}
		\lim\limits_{|x|\to \infty} n(x). \nabla \log\lstat_s(x) = -\infty
	\end{equation}
	and 
	\begin{equation}
		\label{eq:superexp2}
		\limsup\limits_{|x|\to\infty} n(x) . m_s(x) <0
	\end{equation}
	where $\nabla$ is the gradient operator in $\R^{\dhid}$, $n(x)=\frac{x}{|x|}$ is the unit vector pointing in the direction of $x$ and $m_s(x)= \frac{\nabla\lstat_s(x)}{|\nabla\lstat_s(x)|}$ is the unit vector in the direction of the gradient of the stationary distribution at point $x$.
\end{description}
We assume also some regularity properties of the stationary distributions with respect to $\ass$.
\begin{itemize}
	\item[\textbf{(B2)}] For all $x\in \Xspace$, the functions $\ass \mapsto \lstat_s$ and $\ass \mapsto \nabla_x \log \lstat_s$ are continuous on $\Sr$.
\end{itemize}

We now define  for some $\beta \in ]0,1[$, $V_s(x) = c_s\lstat_s(x)^{-\beta}$ where $c_s$ is a constant so that $V_s(x)\geq 1$ for all $x\in \Xspace$. 
Let also $V_1(x)= \inf\limits_{s\in\Sr} V_s(x)$ and $V_2(x)= \sup\limits_{s\in\Sr} V_s(x)$.

Let us assume conditions on $V_2$:
\begin{itemize}
	\item[\textbf{(B3)}] There exists $b_0>0$ such that, for all $s\in\Sr$ and $x\in\Xspace$, $V_2^{b_0}$  is integrable against $\ntrans_s(x,.)$ and  
	\begin{equation}
			\limsup\limits_{b\to0}\sup\limits_{s\in\Sr, x\in \Xspace}\ntrans_s V_2^b(x)  =1\,.
	\end{equation}
\end{itemize}

\begin{prop}\label{ergogeo}
	
	Assume (\textbf{B1-B3}). Let $\Kapa$ a compact subset of $\Sr$.  
	There exist a function $V\geq 1$, a set $\smallset \subseteq \Xspace $, a probability measure $\nu$ such that $\nu(\smallset)>0$ and there exist constants $\lambda \in ]0,1[, \ b\in[0,\infty[$ and $\eps \in ]0,1]$ such that for all $s\in\Kapa$~:
	\begin{eqnarray}
		\label{eq:smallset}
		\ntrans_s(x,A) &\geq& \eps \nu(A) \ \ \ \forall x \in \smallset \ \, \ \forall A \in \mathcal{B}\,,  \\
		\label{eq:drift}
		\ntrans_s V(x) &\leq &\lambda V(x) + b\mathds{1}_\smallset (x)\,.
	\end{eqnarray}
\end{prop}
The proof of Proposition \ref{ergogeo} is given in Appendix. \\

The first equation defines $\smallset$ as a small set for the transition kernels $(\ntrans_s)$. Note that both $\eps$ and $\nu$ can depend on $\smallset$. 
The $ \nu$-small set Equation~\eqref{eq:smallset} "in one step" also implies the $\nu$-irredu- cibility of the transition kernels and their aperiodicity (see \cite{meyntweedie}). 
The second inequality is a drift condition which states that the transition kernels tend to bring back  elements into the small set. 
As a consequence of these well known drift conditions, the transition kernels $(\ntrans_s)$ are $V$-uniformly ergodic. Moreover this property holds uniformly in $s$ in any compact subset $\Kapa\subset \Sr$. That is to say: for any compact $\Kapa\subset \Sr$, there exist $0<\rho<1$ and $0<c<\infty$ such that for all $n\in \N^*$ and $f$ such that $\|f\|_V = \sup\limits_{ x \in \mathcal{X} } \frac{\|f(x)\|}{V(x)}<\infty$:
\begin{equation}
	\sup\limits_{s\in\Kapa} \|\ntrans_s^n f (.)-\lstat_s f\|_V \leq c\rho^n \|f\|_V \,.
\end{equation}

\begin{rem}
	\label{rem:ergoPuissancep}
	The same property holds for any power $p$ of the function $V$ such that $0<p\beta<1$. Indeed, the proof follows the same lines as it can be seen in Section  \ref{sec:proofs}. This is a property that will appear useful in the sequel to prove some properties of the estimation algorithm.
\end{rem}


\subsection{Convergence property of the estimated sequence generated by the AMALA-SAEM algorithm}

We do the following assumptions on the model which are quite usual in the context of 
missing data model using  EM-like  algorithms (see \cite{DLM}, \cite{kuhnlavielle}).

For sake of simplicity we denote in the sequel $p_\te(\cdot)$ instead of 
$p(\cdot|\obs;\te)$ the posterior distribution.

\begin{itemize}
\item {(\bf{M1})} The parameter space $\Te$ is an open subset of $\Rset{\dte}$.
The complete data likelihood function is given by:
\begin{equation*} 
f(\obs,\hid;\te)
= \exp\left[-\psi(\te) + \pscal{\tS(\hid)}{\phi(\te)}\right],
\end{equation*}
where   $\tS$ is a 
Borel function on $\Rset{\dhid}$  taking its values in an open convex subset 
$\Sr$ 
of $\Rset{\dss}$. Moreover, the convex hull of $\tS(\Rset{\dhid})$ is included in 
$\Sr$, and, for all $\te$ in $\Te$,
\begin{equation*} 
\int || \tS(\hid) || p_\te(\hid) \mu(d\hid) < \infty.
\end{equation*}
\item {(\bf{M2})} 
 The functions $\psi$ and $\phi$ are twice 
continuously differentiable on $\Te$. 
\item {(\bf{M3})} 
The function $\bar{s} : \Te \rightarrow \Sr$ defined as
\begin{equation*} 
\bar{s}(\te) \triangleq \int \tS(\hid) p_\te(\hid) \mu(d\hid)
\end{equation*} 
is continuously differentiable on $\Te$.
\item {(\bf{M4})} The function $l : \Te \rightarrow \Rset{}$ defined as the observed-data log-likelihood
\begin{equation*} 
\label{l}
l(\te) \triangleq \log g(y;\te) = \log \int f(\obs,\hid;\te) \mu(d\hid)
\end{equation*} 
is continuously differentiable on $\Te$ and
\begin{equation*} 
\partial_\te \int f(\obs,\hid; \te) \mu(d\hid)= \int \partial_\te f(\obs,\hid; \te)  \mu(d\hid).
\end{equation*} 
\item {(\bf{M5})} There exists a function
$\hte : \ \Sr \rightarrow \Te$, such that:
\begin{equation*} 
\forall s \in \Sr, \ \  \forall \te \in \Te, \ \ 
L(s; \hte(s))\geq L(s; \te).
\end{equation*} 
Moreover,  the function $\hte$ is continuously differentiable 
on $\Sr$. 
\item {(\bf{M6})}      The functions $l:\Theta \to \R$ and $\hte:\Sr  \to \Theta$ are
 $\dss$ times differentiable. 
\end{itemize}   

\begin{itemize}
\item {(\bf{M7})} 
\begin{itemize}
\item[(i)] There exists an $M_0>0$ such that
\begin{equation*}
\left\{
s \in \Sr , \partial_s l(\hte(s))=0 \right\}
\subset \{
 s \in \Sr, \   -l(\hte(s)) < M_0
\}\,.
\end{equation*}
\item[(ii)] For all $M_1 > M_0$, the set $\overline{Conv(\tS(\Rset{\dhid}))} \cap \{
s \in \Sr , \ -l(\hte(s)) \leq M_1\}$
 is a compact set of $\Sr$.
\end{itemize}
\item {(\bf{M8})} 
There exists a polynomial function $P$ such that for all $\hid \in \mathcal{X}$
$$
||\tS(\hid)|| \leq P(\hid)\,.
$$
\item{(\bf{B4})} For any compact subset $\Kapa$ of $\Sr$, there exists a polynomial function $Q$ of the hidden variable such that $\sup\limits_{s\in\Kapa} | \nabla_z \log p_{\hte(s)}(z)| \leq Q(z)  $.
\end{itemize}
Moreover a usual additional assumption is required on the step size sequences of the stochastic approximation.
\begin{itemize}
	\item{\bf(A4)} The sequences $\boldsymbol\gamma=(\gamma_k)_{k\geq 0} $ and
	  $\boldsymbol\varepsilon=(\varepsilon_k)_{k\geq 0} $   are non-increasing,
	  positive and satisfy: there exist $0<a<1$ and $p\geq 2$ such that
	$\sum\limits_{k=0}^\infty \gamma_k =\infty$,
	$\lim\limits_{k\to\infty} \varepsilon_k =0$ and
	
	$ \sum\limits_{k=1}^\infty \{
	\gamma_k^2 +
	\gamma_k \varepsilon_k^a +
	(\gamma_k \varepsilon_k^{-1})^p  \}
	 <\infty$.
\end{itemize}

\begin{theorem}[Convergence Result for the Estimated Sequence generated by Algorithm \ref{algo:AMALASAEM}]
	\label{Th:convergenceAlgo2}
  Assume (\textbf{M1-M8}) and
  (\textbf{A4}). 
Assume that the family of posterior density functions $\{p_{\hte(s)} ,\ s \in \Sr\}$ satisfies
assumptions  (\textbf{B1-B4}).

Let $\mathrm{K}$ be a compact subset of $ \mathcal{X}$
and
$\Kapa_0 \subset \{
 s \in \Sr, \   -l(\hte(s)) < M_0
\} \cap \overline{Conv(\tS(\Rset{\dhid}))}$ (where $M_0$ is defined
in (\textbf{M7})). Then, for all $\hid_0 \in \mathrm{K}$ and
    $s_0\in \Kapa_0$, we have $\lim\limits_{k\to\infty} d(\te_k,\mathcal{L})=0$ a.s. where
$(\te_k)_k$ is the sequence generated by Algorithm \ref{algo:AMALASAEM} and  $\mathcal{L}\triangleq \{\te \in \Te, \partial_\te l(\te)=0\}$.  
\end{theorem}

The proof is postponed to Appendix \ref{subsec:Convergence}.

\subsection{Central Limit Theorem for the estimated sequence generated by the AMALA-SAEM}
\label{subsec:TCL}

Theorem \ref{Th:convergenceAlgo2} ensures that the number of re-initiali- zations of the sequence of stochastic approximation of Algorithm \ref{algo:AMALASAEM} is finite almost surely.  We can therefore consider only the non truncated sequence when we are interested in its asymptotic behavior. 

Let us write the 
stochastic approximation procedure :
$$
s_k=s_{k-1}+\gamma_k h(s_{k-1})+ \gamma_k \eta_k
$$
where $H_s(z)=S(z)-s$, $h(s) = \mathbb{E}_{p_{\hte(s)}}(H_s(z)) $, $\eta_k=\tS(\hid_k) - \mathbb{E}_{p_{\hte(s_{k-1})}} (\tS(\hid))$ and $\mathbb{E}_{p_{\hte(s)}}$ is the expectation under the invariant measure $p_{\hte(s)}$.

Let us introduce some usual assumptions in the spirit of these of Delyon (see \cite{DelyonCours}).

\begin{itemize}
	\item[\textbf{(N1)}]  The sequence $(s_k)_k$ converges to $s^*$ a.s.  The function $h$ is $C^1$ in some neighborhood of $s^*$ with first derivatives Lipschitz and $J$ the Jacobean matrix of the mean field $h$ in $s^*$ has all its eigenvalues with negative real part.  
		\item[\textbf{(N2)}]  
		Let $g_{\hte(s)}$ be a solution of the Poisson equation $g-\ntrans_{\hte(s)} g = H_s - p_{\hte(s)} (H_s)$ for any $s\in\Sr$. 
		There exists a bounded function $w$ such that 
		
		\begin{multline}
			w-\ntrans_{\hte(s^*)}w = g_{\hte(s^*)}g_{\hte(s^*)}^T - \\
			\ntrans_{\hte(s^*)}g_{\hte(s^*)}(\ntrans_{\hte(s^*)}g_{\hte(s^*)})^T - U
			\end{multline} where the deterministic matrix $U$ is given by~:
		\begin{multline}
			\label{eq:defU}
			U= \mathbb{E}_{\hte(s^*)}\left[ g_{\hte(s^*)}(z)g_{\hte(s^*)}(z)^T - \right.\\ 
			\left. \ntrans_{\hte(s^*)}g_{\hte(s^*)}(z)\ntrans_{\hte(s^*)}g_{\hte(s^*)}(z)^T  \right]\,.
		\end{multline} 
		
	\item[\textbf{(N3)}] The step size sequence $(\gamma_k)$  is decreasing and satisfies 
	$\gamma_k = 1/k^\alpha $ with $2/3<\alpha<1$.
	
\end{itemize}
\hspace{1cm}

\begin{theorem}
	\label{Th:TCL}
	Under the assumptions of Theorem \ref{Th:convergenceAlgo2} and under \textbf{(N1)-(N3)}, the sequence \\
	$(s_k-s^*)/\sqrt{\gamma_k}$ converges in distribution to a Gaussian random vector with zero mean and covariance matrix  $\Gamma$ where
	$\Gamma$ is the solution of the following Lyapunov equation: 
	 $$U+J\Gamma +\Gamma J^T =0.$$
	Moreover,
	$$
	\frac{1}{\sqrt{\gamma_k}}(\te_k-\te^*)\to_{\mathcal{L}}\mathcal{N}(0,\partial_s \hte(s^*) \Gamma\partial_s \hte(s^*)^T)
	$$
	where $\theta^* = \hte(s^*)$.
\end{theorem}

The proof of Theorem \ref{Th:TCL} is given in Appendix \ref{appendix:TCL}.

\section{Applications on Bayesian Mixed Effect Template model} 
\label{sec:appli}
\subsection{Comparison between MALA and AMALA samplers} 
\label{sub:comparison_between_mala_and_amala_samplers}
As a first experiment, we compare the mixing properties of  MALA and  AMALA samplers. We used both algorithms to sample from a $10$ dimensional normal distribution with zero mean and non diagonal covariance matrix. Its eigenvalues range from $1$ to $10$. The eigen-directions are chosen randomly. The autocorrelations of both chains are plotted in Fig.~\ref{fig:ComparisonMALA_AMALA} where we can see that there is a benefit of using the anisotropic sampler. 
To evaluate the weight of the anisotropic term  $D(x)D(x)^T$ in the covariance matrix, we compute its amplitude (as its non zero eigenvalue since it is a rank one matrix). We see that it is of the same order as the diagonal part in average and jumps up to $15$ times bigger.
 This shows the importance of the anisotropic term. 
The last check is the Mean Square Euclidean Jump Distance (MSEJD) which computes the expected squared distance between successive draws of the Markov chain. The two methods provide MSEJD of the same order showing a very slight advantage in term of visiting the space for the AMALA sampler ($1.29$ versus $1.25$ for the MALA).

\begin{figure*}[\begin{figure}[htbp]
	\centering
		\includegraphics[height=3cm,width=\textwidth]{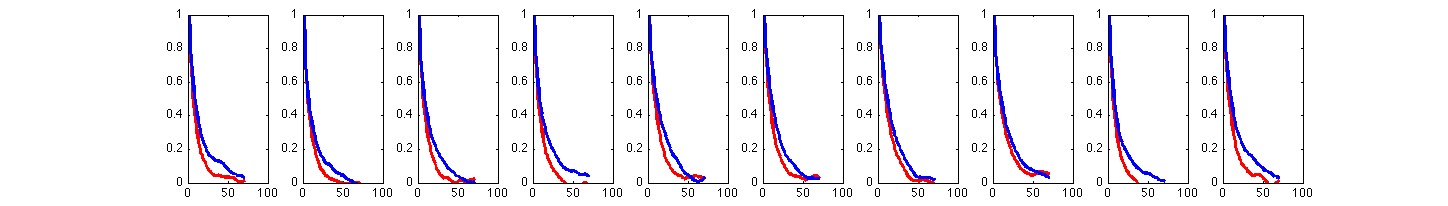}
	\caption{Autocorrelations of the MALA (blue) and AMALA (red) samplers to target the $10$ dimensional normal distribution with anisotropic covariance matrix.}
	\label{fig:ComparisonMALA_AMALA}
\end{figure*}

We will observe in the following experiments that the advantage of considering the AMALA instead of the MALA sampler will be intensified when increasing the problem dimension and including it into our estimation process.


\subsection{BME Template estimation} 
\label{sub:bme_template_estimation}

Back to our targeted application, we apply the proposed estimation process on different data bases. The first one is the USPS hand-written digit base as used in \cite{AAT} and \cite{aktdefmod}. The other two are medical images of 2D corpus callosum and 3D murine dendrite spine excrescences used in \cite{aktSFDS2010}.\\

We begin with presenting the experiments on the USPS database. In order to make comparison, we estimate the parameters in the same conditions as in the previous mentioned works that is to say using the same $20$ images per digit. Each image has grey level between $0$ (background) and $2$ (bright white). These images are presented on the top panel of Fig.~\ref{fig:pics_trainingset}. We also use a noisy training dataset generated by adding a standardized independent Gaussian noise. These images are presented on the bottom panel of Fig.~\ref{fig:pics_trainingset}.
We test five algorithms: the deterministic approximation of the EM algorithm  (FAM-EM) presented in \cite{AAT}, four MCMC-SAEM where the sampler is either the MALA, the adaptive MALA proposed in \cite{atchade2006}, the hybrid Gibbs sampler presented in \cite{aktdefmod} and our AMALA algorithm.
\begin{figure}[htbp]\begin{center}
\includegraphics[width=6cm,height=5cm]{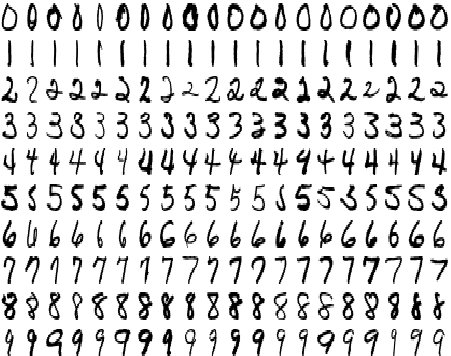}\\
\vspace{0.5cm}
\includegraphics[width=6cm,height=5cm]{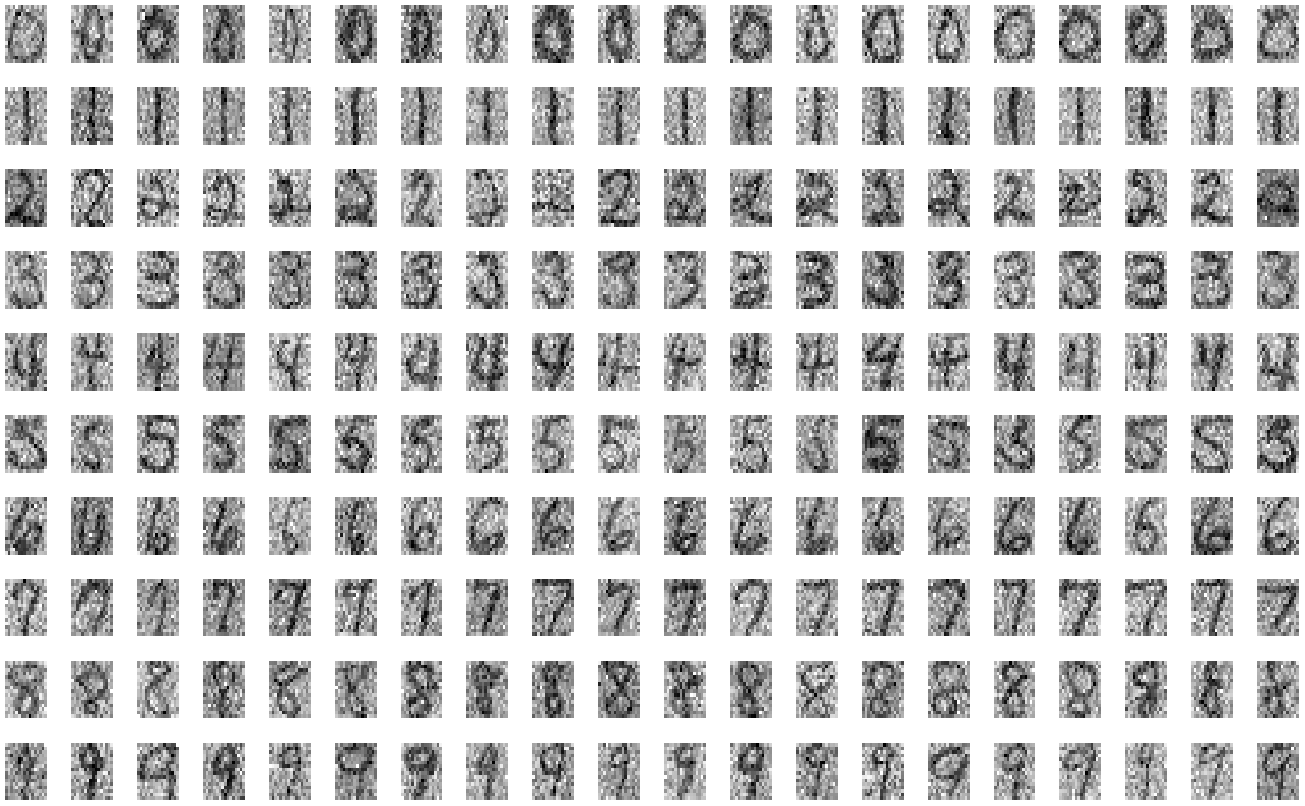}
\end{center}
\caption{Top: twenty images per digit of the training set used for the estimation of the model
  parameters (inverse video). Bottom: same images with additive noise of variance $1$.}
\label{fig:pics_trainingset}
\end{figure}

For these experiments the tuning parameters are chosen as follows: the threshold $b$ is set to $1,000$, the scale $\delta$ to $10^{-3}$ and the regularization $\varepsilon$ to $10^{-4}$. The other tuning parameters and hyper-parameters are chosen as in \cite{aktdefmod}.

Note that this model satisfies the conditions of our convergence theorem as these conditions are similar to the ones proved in \cite{aktdefmod}. 

\subsection{Computational performances} 
\label{sub:computational_performances}

We compare first the computational performances of the algorithms.
The computational time is smaller for the three MCMC-SAEM algorithms using "MALA-like" samplers compared to the FAM. 
Indeed, a numerical convergence of that algorithm requires about $30$ to $50$ EM steps. Each of them requires a gradient descent which has $15$ iterations in average. This implies to compute $15$ times the gradient of the energy (which actually equals our gradient) for each image for each EM step. The "MALA-like"-SAEM algorithms require about $100$ to $150$ EM steps (depending on the digit) but only one gradient is computed for each image at each step. This reduces the computational time by a factor of at least $4$ (up to $7$ depending on the digit). No comparison can be done when the data are noisy since the FAM-EM does not converges toward the MAP estimator as mentioned above.
Comparing to the hybrid Gibbs-SAEM, the computational time is $8$ times lower with the AMALA-SAEM in this particular case of application. Indeed, the hybrid Gibbs sampler requires no computation of the gradient. However, it includes a loop over the coordinates of the hidden variable, here the deformation vector of size $2k_g=72$. At each of these iterations, the candidate is straightforward to sample whereas the computational cost lies into the acceptance rate. When this becomes heavy, the less times you calculate it, the better. In the AMALA-SAEM, this acceptance rate only has to be calculated once for each image. Therefore, even when the dimension of the hidden variable increases, this is of constant cost. The main price to pay is the computation of the gradient. Therefore, a tradeoff has to be found between the computation of either one gradient or $dk_g$ acceptance rates in order to select the algorithm to use in a given case.


\subsection{Results on the template estimation}

All the estimated templates obtained with the five algorithms and noise-free and noisy training data are presented in Fig.~\ref{fig:pics_templatesOfDigs}.
As noticed in \cite{aktdefmod}, the FAM-EM estimation is sharp when the training set is noise-free and is deteriorated while adding noise. This behavior is not surprising with regard to the theoretical bound established in \cite{bc2011} in the particular case of compact deformation group.
Considering the adaptive sampler, it does not reach a good estimation of the templates which are still very blurry and noisy in both cases. The problem seems to come from the very low acceptation rate already at the beginning of the estimation. The bad initial guess we have about the covariance matrix of the proposal seems to block the chain. Moreover, the tuning parameters are difficult to calibrate along the iterations of the estimation algorithm. 
Concerning the estimated templates using the Gibbs, MALA and  AMALA samplers, they look very similar to each other using the noise-free data as well as the noisy ones. This similarity confirms the convergence of all these algorithms toward the MAP estimator. In this case, the templates are as expected: noise free and sharp.

Nevertheless, when the dimension of the hidden variable increases, both the Gibbs and the MALA samplers show limitations. We run the estimation on the same noisy USPS database, increasing the number $k_g$ of geometrical control points. We choose the dimension of the deformation vector equal to $72$, $128$ and $200$. 
The Gibbs-SAEM would produce sharp estimations but explodes the computational time. For this reason, we did not run this algorithm on higher dimension experiments. The results are presented in Fig.~\ref{fig:pics_templateskghigh}.
Concerning the MALA sampler, it does not seem to capture the whole variability of the population in such high dimension. This yields a poorly estimation of the templates. This phenomenon does not appear using our AMALA-SAEM algorithm. The templates still look sharp and the acceptation rates remain reasonable.

\begin{figure*}[htbp]
	\centering
	\begin{tabular}{c|c|c|c|c|c}
		\raisebox{0cm}{\parbox{1.5cm}{\centering Algo./\\Noise}}&	FAM & Hybrid  Gibbs  & MALA &  Adaptive MALA  & AMALA \\\\
		\hline
		\\
		\raisebox{1cm}{\parbox{1.5cm}{\centering No \\ Noise}}&
		\includegraphics[height=2cm]{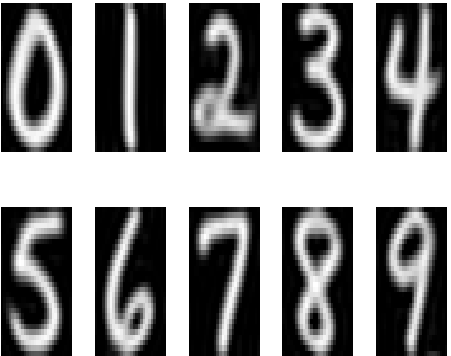}&
		\includegraphics[height=2cm]{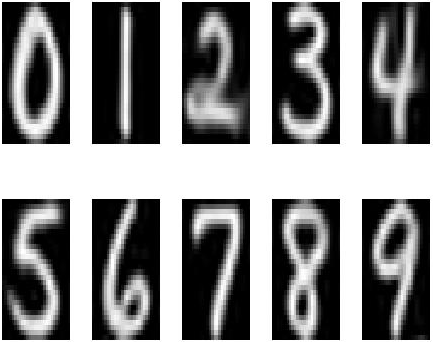}&
		\includegraphics[height=2cm]{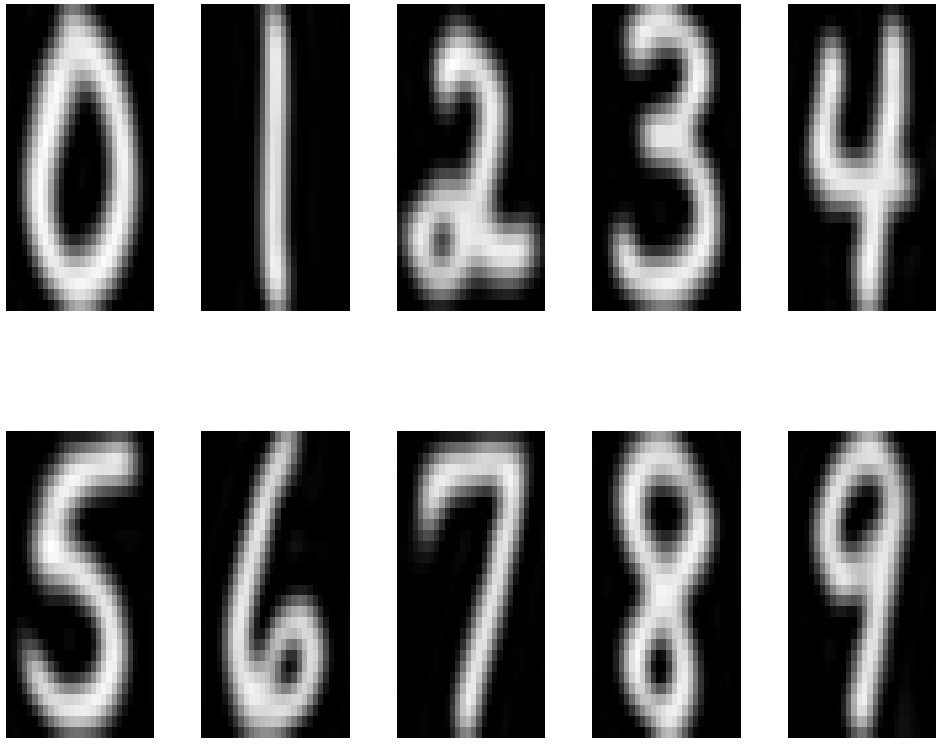}&
		\includegraphics[height=2cm]{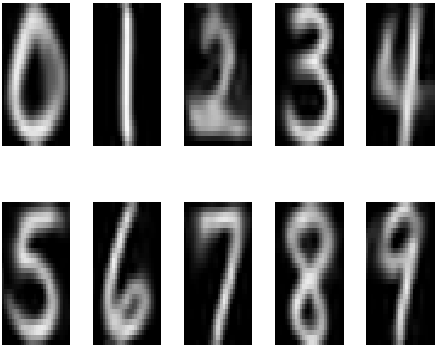}& 
		\includegraphics[height=2cm]{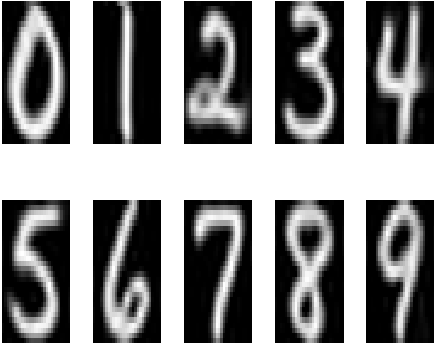}\\		
		\hline
		\\
		\raisebox{1cm}{\parbox{1.5cm}{\centering  Noise\\ }}&
		\includegraphics[height=2cm]{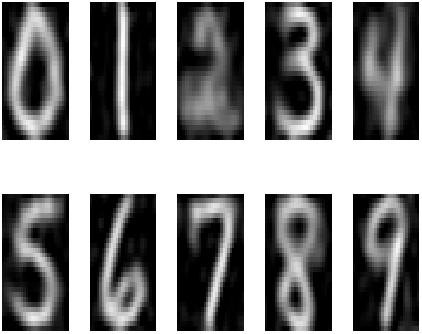}&
		\includegraphics[height=2cm]{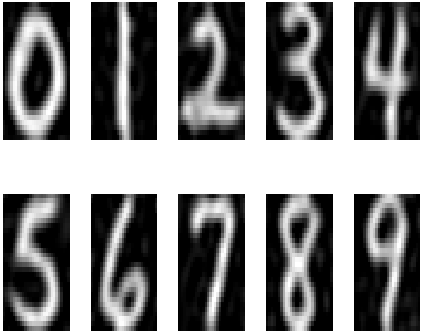}&
		\includegraphics[height=2cm]{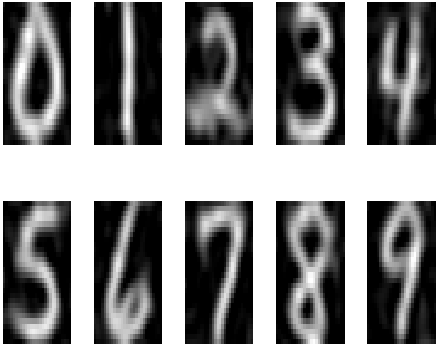}&
		\includegraphics[height=2cm]{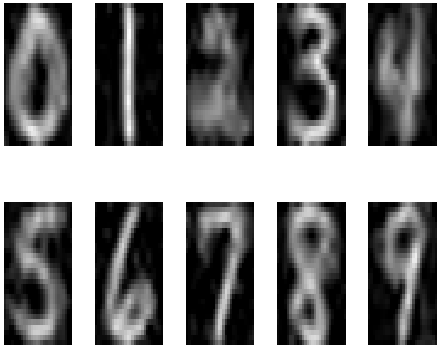}& 
		\includegraphics[height=2cm]{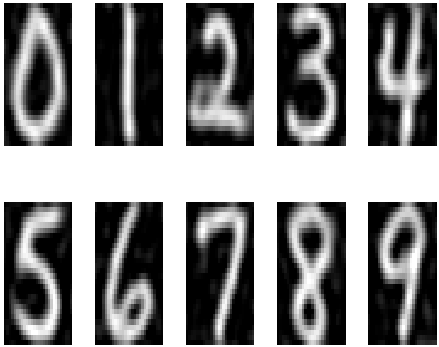}\\
		\hline
	\end{tabular}
	\caption{Estimated templates using the five algorithms on  noise free and noisy data. The training set includes $20$ images per digit. The dimension of the hidden variable is $72$.}
	\label{fig:pics_templatesOfDigs}
\end{figure*}

\begin{figure*}[htbp]
	\centering
	\begin{tabular}{c|c|c|c}
		\raisebox{0cm}{\parbox{1.5cm}{Dim. of def. / Sampler}} & $2k_g=72$ & $2k_g=128$ & $2k_g=200$ \\
		\hline
		\\
		 \raisebox{1cm}{\parbox{1.5cm}{\centering MALA\\ }} &
		\includegraphics[height=2cm]{templateMALA18.png}&
		\includegraphics[height=2cm]{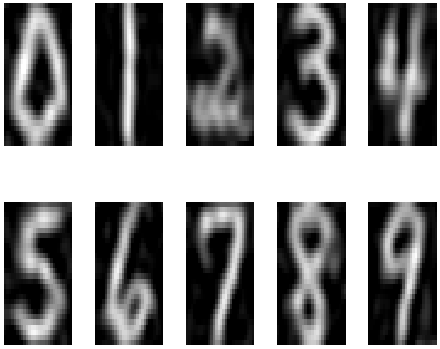}&
		\includegraphics[height=2cm]{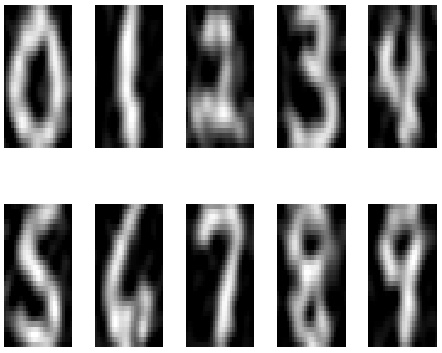}\\
		\hline
		\\
		\raisebox{1cm}{\parbox{1.5cm}{\centering AMALA\\ }} &
		\includegraphics[height=2cm]{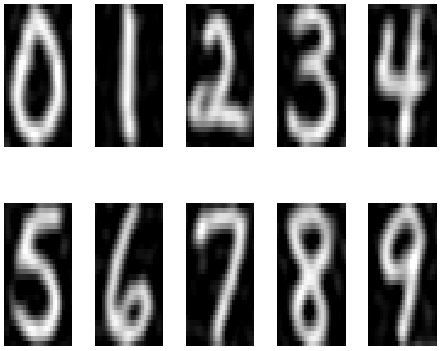}&
		\includegraphics[height=2cm]{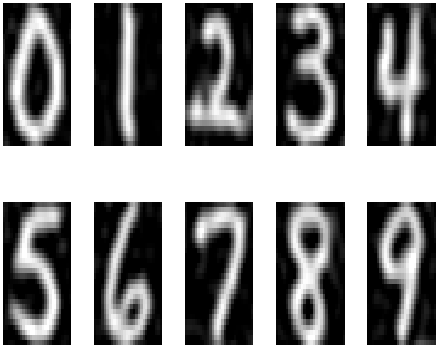}&
		\includegraphics[height=2cm]{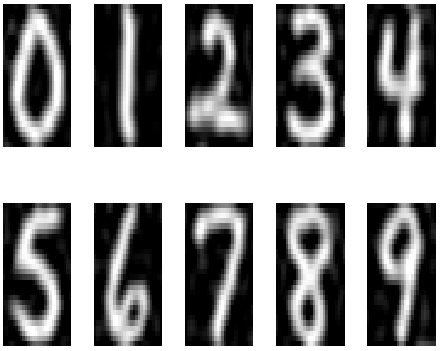}\\
		\hline
	\end{tabular}
	\caption{Estimated templates using MALA and AMALA samplers in the stochastic EM algorithm on noisy training data. The training set includes $20$ images per digit. The dimension of the hidden variable increases from $72$ to $200$.}
	\label{fig:pics_templateskghigh}
\end{figure*}

\subsection{Results on the covariance matrix estimation}
\label{subsec:covar}

Since we are provided with a generative model, once the parameters have been estimated, we can generate synthetic samples in order to evaluate the constrained on the deformations that have been learnt. Some of these samples are presented in Fig.~\ref{fig:pics_SamplesOfDigs}. For each digit, $20$ examples are generated with the deformations given by $+z$ and $20$ others with $-z$ where $z$ is simulated with respect to $\mathcal{N}(0,\Gamma_g)$. We recall that, as already noticed in \cite{aktdefmod}, the Gaussian distribution is symmetric which may lead to strange samples in one direction whereas the other one looks like something present in the training set.

With regards to the above remarks concerning the computational time and the template estimations, we present in this subsection only the results obtained using MALA and AMALA-SAEM algorithms. We notice that the samples generated by both algorithms look alike in the case of hidden variable of dimension $72$. Thus, we present only the results of our AMALA-SAEM estimation.
As we can see, the deformations are very well estimated in both cases (without or with noise) and even look similar. This tends to demonstrate that the noise has been separated from the template as well as the geometric variability during the estimation process.
\begin{figure*}[htbp]
	\centering
		\includegraphics[height=7cm,width=5cm]{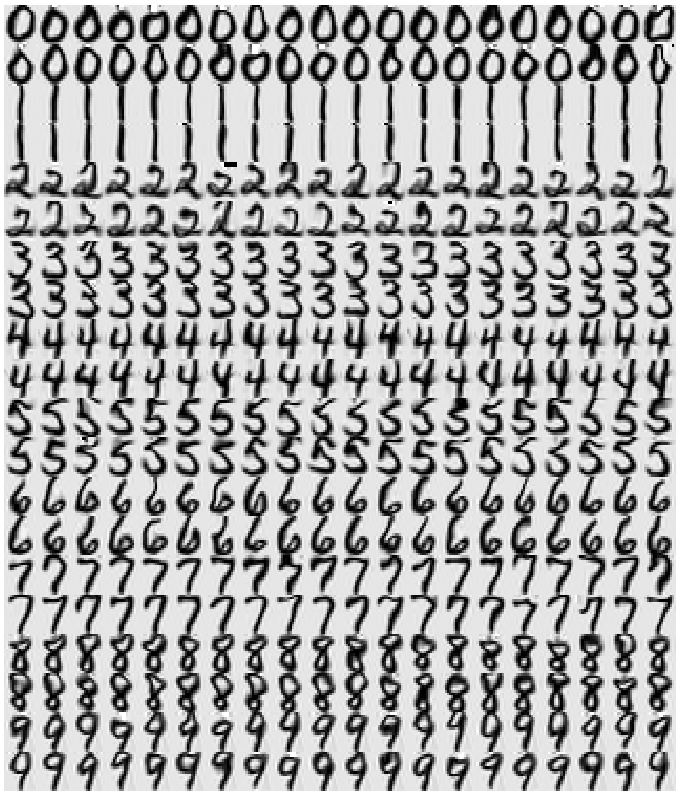}
		\includegraphics[height=7cm,width=5cm]{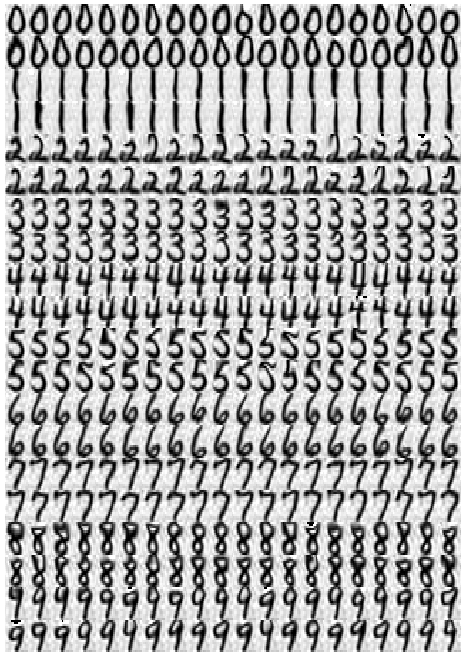}
	\caption{Synthetic samples generated with respect to the BME template model using the estimated parameters with AMALA-SAEM. For each digit, the two lines represent the deformation using $+$ and $-$ the simulated deformation $z$. Left: data without noise. Right: data with noise variance $1$. The number of geometric control points is $36$ leading to a hidden variable of dimension $72$.}
	\label{fig:pics_SamplesOfDigs}
\end{figure*}

Increasing the dimension of the deformation to $128$, we run both algorithms on the noisy dataset. We observe on Fig.~\ref{fig:pics_SamplesOfDigsHighDim} that the geometric variability of the samples remains similar to the one obtained in lower dimension using our AMALA-SAEM. However, the MALA-SAEM does not manage to capture the whole variability of the deformations which is related to the results observed above on the template. This confirms the limitation of the use of MALA-SAEM in higher dimension.

\begin{figure*}[htbp]
	\centering
		\includegraphics[height=7cm,width=5cm]{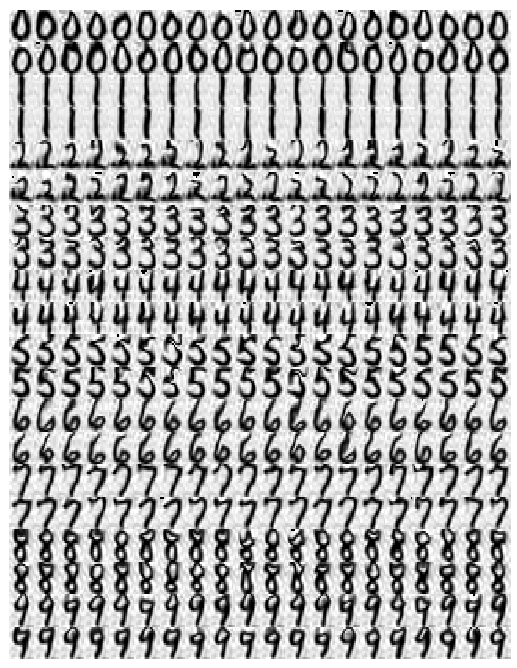}
		\includegraphics[height=7cm,width=5cm]{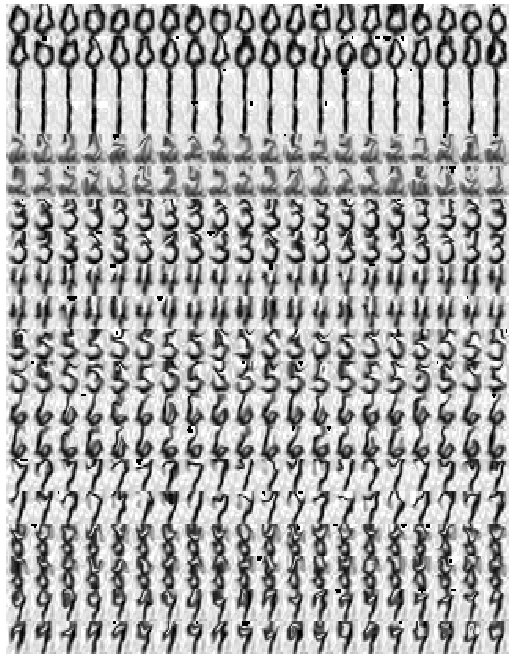}
	\caption{Synthetic samples generated with respect to the BME template model using the estimated parameters with AMALA-SAEM (left) and MALA-SAEM (right). For each digit, the two lines represent the deformation using $+$ and $-$ the simulated deformation $z$. The number of geometric control points is $64$ leading to a hidden variable of dimension $128$.}
	\label{fig:pics_SamplesOfDigsHighDim}
\end{figure*}

\subsection{Results on the noise variance estimation}

The last check of the accuracy of the estimation relies in the noise variance estimation. The plots of their evolutions along the AMALA-SAEM iterations for each digit in both cases (without and with noise) are presented in Fig.~\ref{fig:pics_evolSigmaDigs}. This variance is underestimated in particular in the noisy case, which is a well-known effect of the maximum likelihood estimator. We observe that  the geometrically very constrained digits as $1$ or $7$ tend to converge very quickly whereas the digits $2$ and $4$ require more iterations to capture all the shape variability. 

Since this is a real parameter, we used it to illustrate the Central Limit Theorem stated in Subsection~\ref{subsec:TCL}. Figure~\ref{fig:pics_HistSigma} and Figure~\ref{fig:pics_HistSigma2} show the histograms of $10,000$ runs of the algorithm with the same initial conditions. 
We use the digits $0$ and $2$ of the original data set as well as of the noisy data.
As the iterations go along, the distribution of the estimates tends to  look like a Gaussian distribution centered in the estimated noise variances which demonstrates empirically the Central Limit Theorem.

\begin{figure}[htbp]
	\centering
	\includegraphics[height=5.5cm]{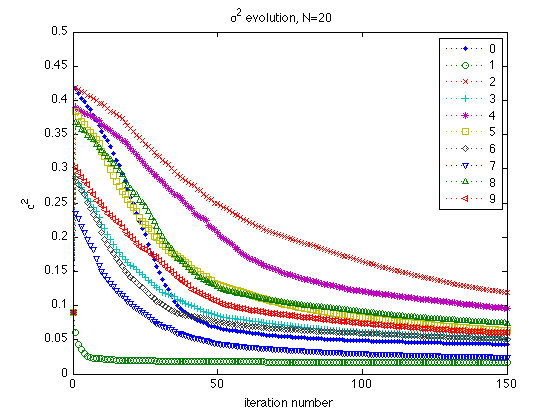}
	\includegraphics[height=5.5cm]{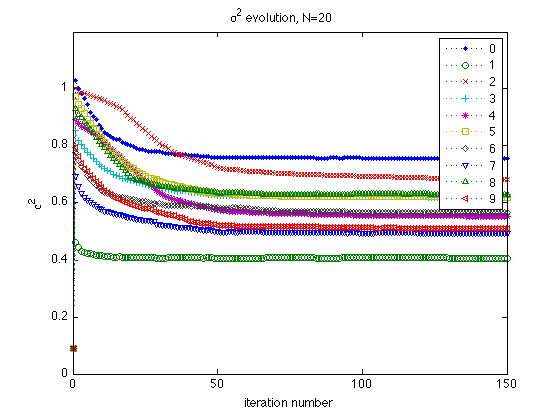}
	\caption{Evolution of the estimation of the noise variance along the AMALA-SAEM iterations. Top: original data. Bottom: noisy data.}
	\label{fig:pics_evolSigmaDigs}
\end{figure}

\begin{figure*}[htbp]
	\centering
	\includegraphics[height=4cm, width=7cm]{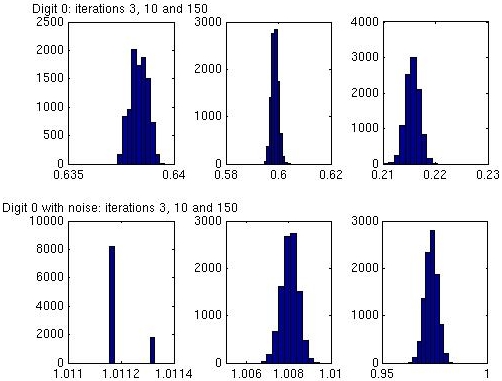} 
	\caption{Empirical convergence toward the Gaussian distribution of the estimated noise variance along the AMALA-SAEM iterations for digit $0$. Top: original data. Bottom: noisy data. }
	\label{fig:pics_HistSigma}
\end{figure*}
\begin{figure*}[htbp]
	\centering
	\includegraphics[height=4cm, width=7cm]{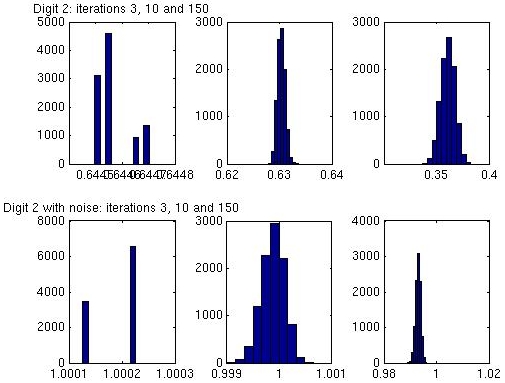}
	\caption{Empirical convergence toward the Gaussian distribution of the estimated noise variance along the AMALA-SAEM iterations for digit $2$. Top: original data. Bottom: noisy data.}
	\label{fig:pics_HistSigma2}
\end{figure*}

\subsection{Classification results} 
\label{sub:classification_results}
The deformable template model enables to perform classification using the maximum likelihood of a new image to allocate it to one class, here the digit.
We use the test USPS database (which contains $2007$ digits) for classification while the training was done on the previous $20$ noisy images. The results obtained with the hybrid Gibbs, MALA and AMALA-SAEM are presented in Table~\ref{table:classif}. In dimension $72$, the best classification rate is performed by the hybrid Gibbs-SAEM. This is easily understandable since the sampling scheme enables to catch deformations which have been optimized control point by control point. Therefore, the estimated covariance matrix carries more local accuracy.
The AMALA-SAEM leading to a much smaller computation time and to estimates of the same quality provides also a very good classification rate. This confirms the good results observed on both the template estimates and the synthetic samples. 
Unfortunately, the MALA-SAEM shows again some limitations. Even if the templates look acceptable, the sampler does not manage to capture the whole class variability. Therefore, the classification rate falls down. 

In order to evaluate the stability of our estimation algorithm with respect to the dimension, we perform the same classification with more control points. As expected, the MALA-SAEM classification rate is deteriorated whereas our AMALA-SAEM keeps very good performances. Note that the hybrid Gibbs sampler was not tested in dimension $2k_g = 128$ because of its very long computational time.

\begin{table}[htbp]
\begin{center}
\begin{tabular}{c | c | c | c}
Sampler  / \\
Dim. of def.   & Hybrid Gibbs & MALA   &  AMALA  \\
 & & &  \\ \hline
$72$ &      $ 22.43  $    &  $ 35.98 $     & $23.22 $  \\
& & & \\
$128$ &   $ \times $ & $43.8$ & $ 25.36$
\end{tabular}
\end{center}
\caption{Error rate using the estimations on the noisy training set with respect to the sampler used in the MCMC-SAEM algorithm and the dimension of the deformation $2k_g$. The classification is performed on the test set of the USPS database. 
}\label{table:classif}
\end{table}

\subsection{2D medical image template estimation}

A second database is used to illustrate our algorithm. As before, in order to make comparisons with existing algorithms, we  use the same database presented in \cite{aktSFDS2010}. It consists of $47$ medical images, each of them is a $2D$ square zone around the end point of the corpus callosum. This box contains a part of this corpus callosum as well as a part of the cerebellum. Ten exemplars are presented in the top rows of Fig.~\ref{fig:SFDS}.

The estimations are compared with these obtained with the FAM-EM and the hybrid Gibbs-SAEM  algorithms and with the grey level mean image (bottom row of Fig.~\ref{fig:SFDS}). In this real situation, the Euclidean grey level mean image (a) is very blurry. The estimated template using the FAM-EM (b) provides a first amelioration in particular leading to a sharper corpus callosum. However, the cerebellum still looks blurry in particular when comparing it to the shape which appears in the template estimated using the hybrid Gibbs SAEM (c). The result of our AMALA-SAEM is given in image (d). This template is very close to (c) as we could expect at a convergence point.
Nevertheless the AMALA-SAEM has much lower computational time than the hybrid Gibbs-SAEM. This shows the advantage of using AMALA-SAEM in real cases of high dimension.

\begin{figure}[h]
\begin{minipage}[h]{1.0\linewidth}
    \begin{tabular}[h]{c}
      \hspace{1cm}\\
       \includegraphics[width=0.8\textwidth]{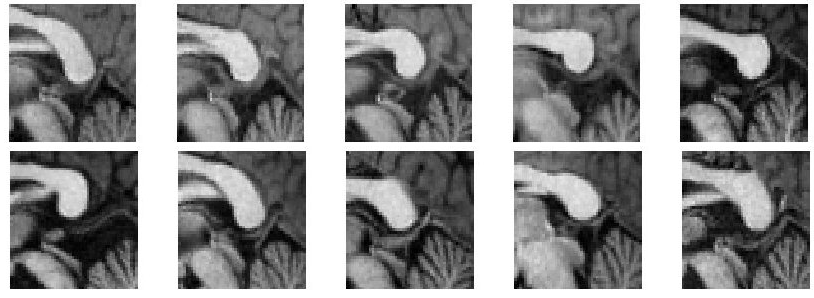}\\
      \begin{tabular}[h]{ccccc}
       \hspace{1cm}   \\
        \includegraphics[width=0.18\textwidth]{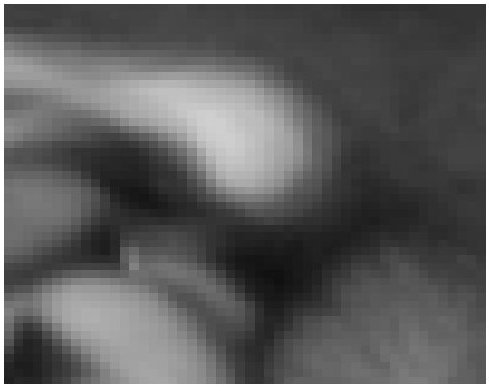}
        &
          \includegraphics[width=0.18\textwidth]{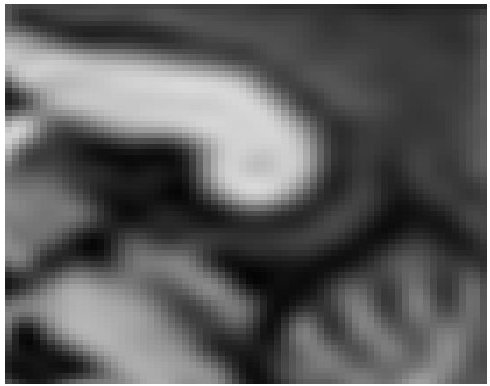}       
        &
          \includegraphics[width=0.18\textwidth]{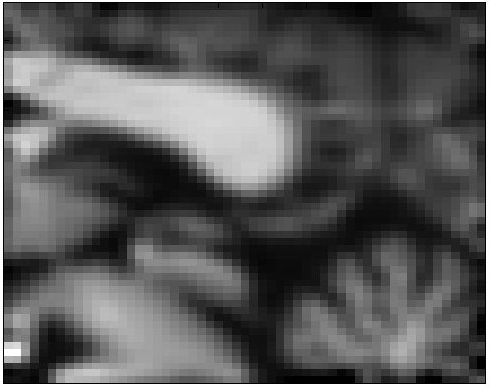}
		&
          \includegraphics[width=0.18\textwidth]{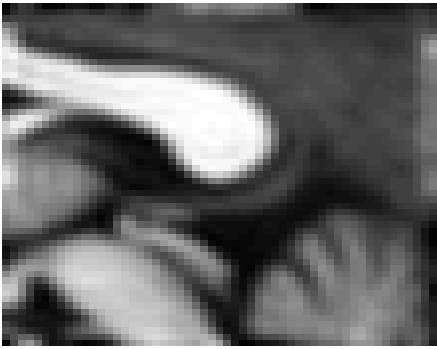} \\
		 (a) & (b) & (c) &  (d)
        \end{tabular}
        \end{tabular} 
\caption{\label{fig:SFDS} Medical image template estimation. Top rows~: $10$ Corpus callosum and cerebellum training images among the $47$ available. Bottom row~: (a) mean image. (b) FAM-EM estimated template. (c) Hybrid Gibbs - SAEM estimated template. (d) AMALA-SAEM estimated template.}   
  \end{minipage}
\end{figure}

\subsection{3D medical image template estimation} 
\label{sub:3d_medical_images}
We also test our algorithm in much higher dimension using the dataset of murine dendrite spines (see \cite{dataDendrite,cfthbmrISPA2007,cofthbmrISHIB2007}) already used in \cite{aktSFDS2010}.  The dataset consists of $50$ binary images of microscopic structures, tiny protuberances found on many types of neurons termed dendrite spines. The images are from control mice and knockout mice which have been genetically modified to mimic human neurological pathologies like Parkinson's disease.
The acquisition process consisted of electron microscopy after injection of Lucifer yellow and subsequent photo-oxidation.
The shapes were then manually segmented on the tomographic reconstruction of the neurons. Some of these binary images are presented in Fig.~\ref{fig:TrainSet} which shows a 3D view of some exemplars among the training set.
Each image is a binary (background $=0$, object $=2$) cubic volume of size $28^3$. We can notice here the large geometrical variability of this population of images. Therefore we use a hidden variable of dimension $3k_g=648$ to catch this complex structure. 

The template estimated with either $30$ or $50$ observations are presented in Fig.~\ref{fig:pics_TemplateDendrite30sujets}. We obtain similar shapes which are coherent with what a mean shape could be regarding the training sample. To evaluate the estimated geometrical variability, we generate synthetic samples as done in Subsection~\ref{subsec:covar}. Eight of these are shown in Fig.~\ref{fig:Sample}. We observe different twisting which are all coherent with the shapes observed in the dataset. Note that the training shapes have very irregular boundaries whereas the parametric model used for the template leads to a smoother image. Thus, the synthetic samples do not reflect the local ruggedness of the segmented murine dendrite spines. If the aim was to capture these local bumps, the number of photometrical control points has to be increased. However, the goal of our study was to detect global shape deformations.
 
\begin{figure*}[tb]
	\centering
	\includegraphics[width=0.7\textwidth]{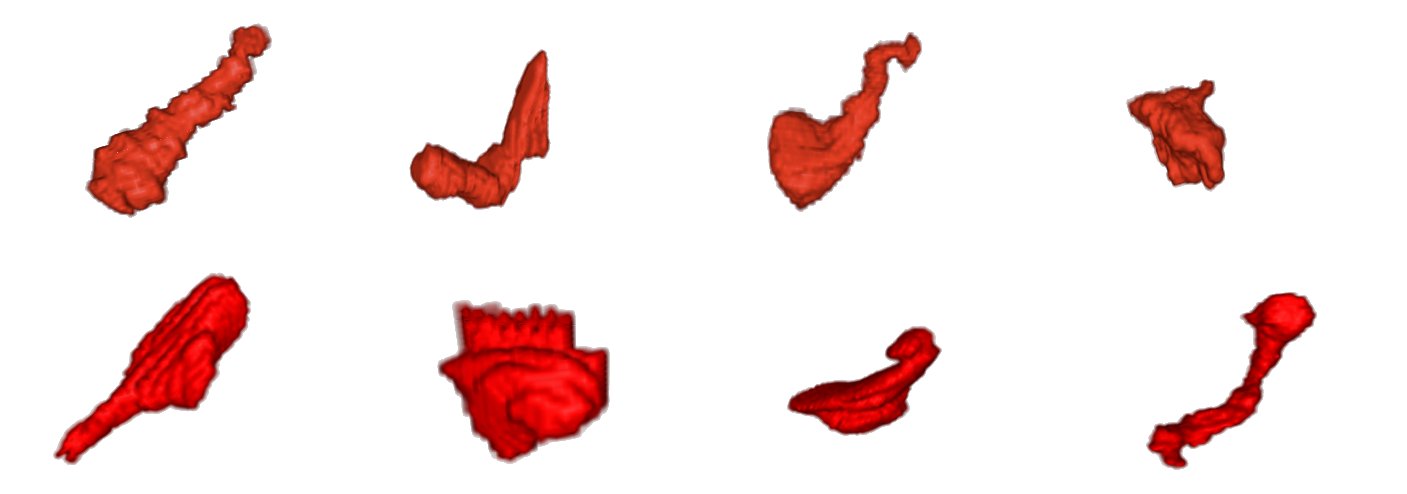}
	\caption{3D views of eight samples of the data set of dendrite spines. Each image is a volume leading to a binary image.}
	\label{fig:TrainSet}
\end{figure*}
\begin{figure*}[bt]
	\centering
	\includegraphics[width=0.5\textwidth]{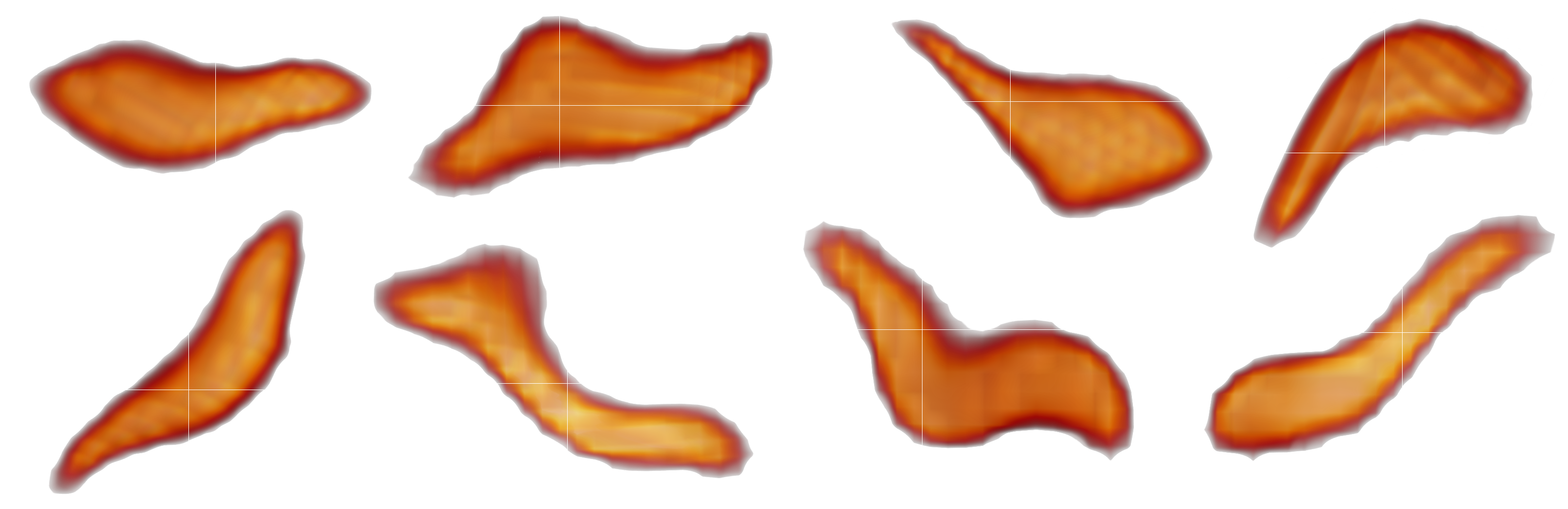}
	\caption{3D views of eight synthetic data. The estimated template shown on the left of Fig.~\ref{fig:pics_TemplateDendrite30sujets} is randomly deformed with respect to the estimated covariance matrix. 
	}
	\label{fig:Sample}
\end{figure*}

\begin{figure}[htbp]
	\centering
		\includegraphics[height=1.5cm]{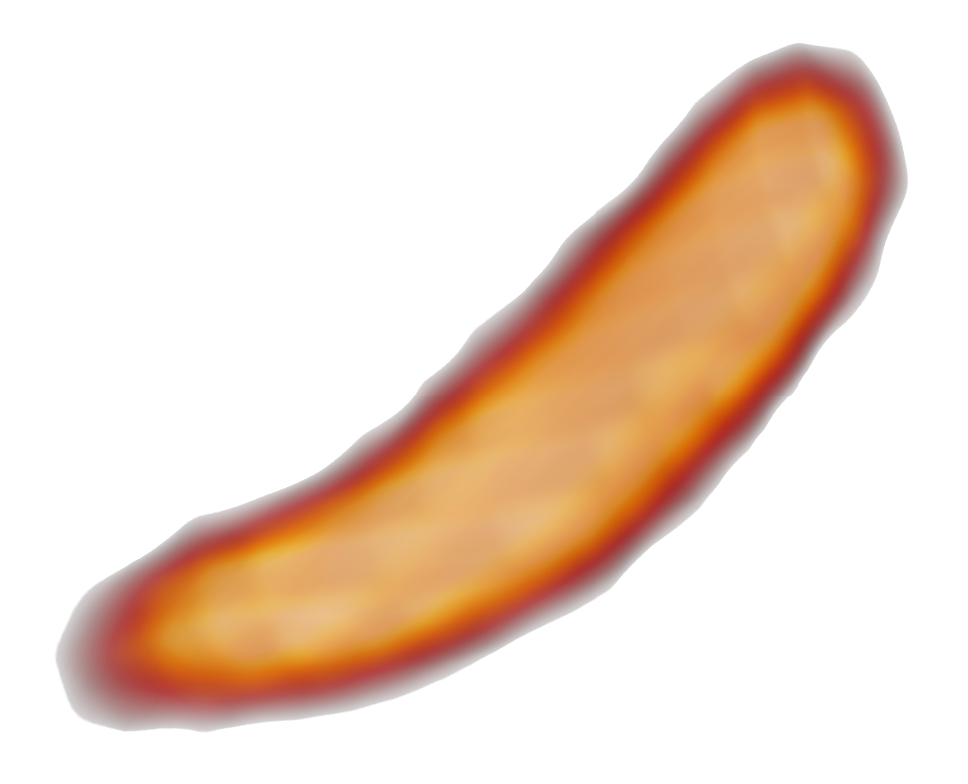} \hspace{1cm}
		\includegraphics[height=1.5cm]{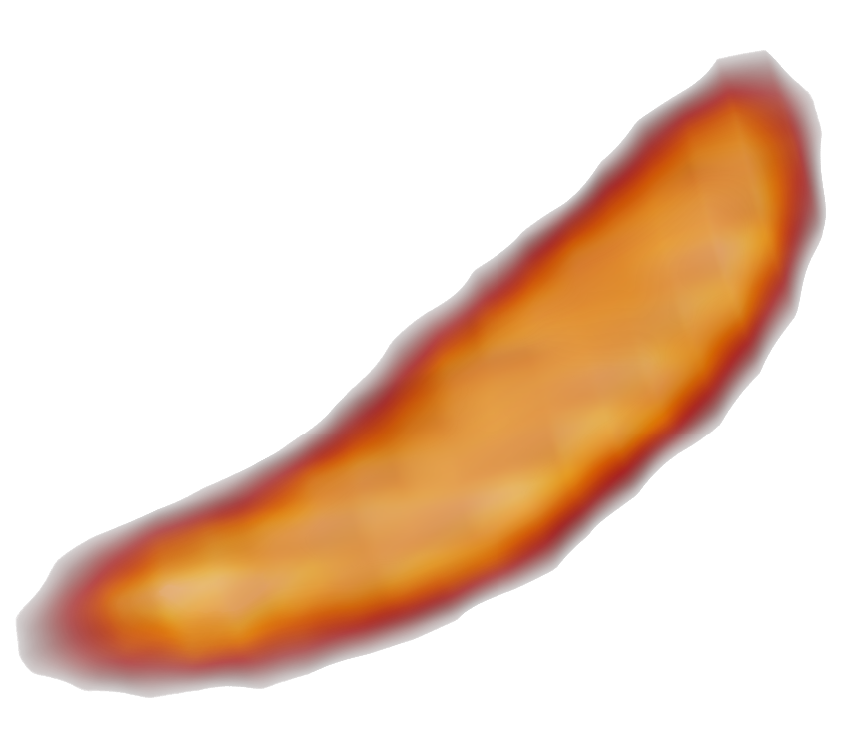}
	\caption{Estimated templates of murine dendrite spines. The training set is either composed of $30$ (left) or $50$ (right) images.}
	\label{fig:pics_TemplateDendrite30sujets}
\end{figure}



\section{Conclusion} 
\label{sec:conclusion}
In this paper we have considered the deformable template estimation issue using the BME model. We were particularly interested in the high dimensional setting. To that purpose, we have proposed to optimize the sampling scheme in the MCMC-SAEM algorithm to get an efficient and accurate estimation process. We have exhibited a new MCMC method based on the classical Metropolis Adjusted Langevin Algorithm where we introduced an anisotropic covariance matrix in the proposal. This  optimization takes into account the anisotropy of the target distribution. We proved that the generated Markov chain is geometrically ergodic uniformly on any compact set. We have also proved the almost sure convergence of the sequence of parameters generated by the estimation algorithm as well as its asymptotic normality.  We have illustrated this estimation algorithm in the BME model. We have considered different datasets of the literature namely the USPS database, 2D medical images of corpus callosum and 3D medical images of murine dendrite excrescences. We have compared the results with previously published ones to highlight the gain in speed and accuracy of the proposed algorithm.

We emphasize that the proposed estimation scheme can be applied in a wide range of application fields involving missing data models in high dimensional setting. In particular, this method is promising when considering mixture models as proposed in \cite{aksaemmulti}. Indeed, it will enable to shorten the computation time of the simulation part which in that case requires the use of many auxiliary Markov chains.
This also provides a good tool for this BME model when introducing a diffeomorphic constrain on the deformations. In this case,
it is even more important to get an efficient estimation process since the computational cost of diffeomorphic deformation is intrinsically large.



\section{Appendix} 
\label{sec:proofs}

\subsection{Proof of Proposition \ref{ergogeo}}
The idea of the proof is the same as the one of the geometric ergodicity of the random walk Metropolis algorithm developed in \cite{jh98} and reworked in \cite{atchade2006} for its adaptive version of the MALA with truncated drift. The fact that both the drift and the covariance matrix are bounded even depending on the gradient of $\log\lstat_s$ enables partially similar proofs.

Let us first recall the transition kernel:
\begin{multline}
	\ntrans_s(x,A) = \int_A \alpha_s(x,z) \cand(x,z)dz + \\
	\mathds{1}_{A}(x) \int_{\Xspace}(1-\alpha_s(x,z))\cand(x,z)dz\,,
\end{multline}
where $\alpha_s(x,z) =\min (1,\rho_s(x,z)) $ and $\rho_s(x,z)= \frac{\lstat_s(z)\cand(z,x)}{\cand(x,z)\lstat_s(x)}$.\\

Thanks to the bounded drift and covariance matrix, we can bound uniformly in $s\in\Sr$ the proposal distribution $\cand$ by two centered Gaussian distributions as follows: there exist  constants $0<k_1<k_2$,  $\epsilon_1 >0$ and $\epsilon_2 >0$ such that for all $(x,z)\in \Xspace^2$ and for all $s\in\Sr$
\begin{equation}
	\label{eq:encadrementCand}
	k_1 g_{\epsilon_1}(x-z) \leq \cand(x,z) \leq	k_2 g_{\epsilon_2}(x-z)\,,
\end{equation}
denoting by $g_a$ the centered Gaussian probability density function in $\R^{\dhid}$ with covariance matrix $aId_{\dhid}$.

\subsubsection{Proof of the existence of a small set $\smallset$}\label{proofsmallset}
Let $\smallset$ be a compact subset of $\Xspace$.

Let  $K$ be a compact set. 
We define $\tau = \inf \{ \rho_s(x,z), \ x\in\smallset, \ z\in K, \ s\in\Kapa\}$. Since $\rho_s$ is a ratio of positive continuous functions in $s,x$ and $z$ and $\Kapa$ is a compact subset of $\Sr$, we have $\tau>0$.  The same argument holds for $(s,x,z)\mapsto \cand(x,z)$ which is bounded by below by $\mu>0$.
Therefore, for all $x\in \smallset$, for any $A \in \mathcal{B}$ and for all $s\in\Kapa$~:

\begin{eqnarray*}
	\ntrans_s(x,A) &\geq& \int_{A\cap K} \alpha_s(x,z)\cand(x,z) dz \\
	&\geq& \min(1,\tau) \mu  \int_{A}\mathds{1}_{K}(z) dz \,.
\end{eqnarray*}

Therefore, we can define $\nu(A)= \frac{1}{Z}\int_A  \mathds{1}_{K}(z)dz$ where $Z$ is the renormalisation constant and $\eps= \min(1,\tau) \mu Z$ so that $\smallset$ is a small set for the transition kernel $\ntrans_s$ for all $s\in\Kapa$ and \eqref{eq:smallset} holds. 

\subsubsection{Proof of the drift condition}

We will prove this property in two steps. First, we establish that each kernel $\ntrans_s$ satisfies a Drift property with a specific function $V_s$. Then, we construct a  common function $V$ so that we will be able to prove the Drift property uniformly in $s\in\Kapa$.\\

Let us concentrate on the first step. Let us consider $s$ fixed. As already suggested in \cite{jh98}, we only need to prove the two following conditions: 
\begin{equation}
	\label{eq:supfini}
	\sup\limits_{x\in\Xspace} \frac{\ntrans_s  V_s(x)}{V_s(x)} <\infty
\end{equation}
and 
\begin{equation}
	\label{eq:limsupinf1}
	\limsup\limits_{|x|\to \infty} \frac{\ntrans_s V_s(x)}{V_s(x)} <1\,.
\end{equation}
We take the same path as in \cite{atchade2006} applied to our case and refer to Fig.~\ref{fig:pics_dessin} for a 2D visualization of all the sets introduced along the proof.

\begin{figure}[htbp]
\includegraphics[height=2.5in]{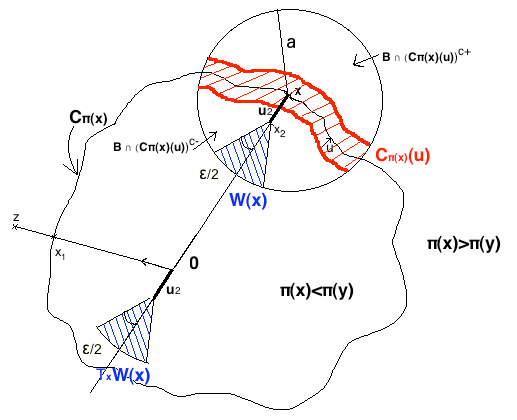}
	\caption{2D representation of the sets used in the proof.}
	\label{fig:pics_dessin}
\end{figure}

For any $x\in \Xspace$, we denote by

\noindent $A_s(x)= \left\{  z \in \Xspace \text{ such that } \rho_s(x,z) \geq 1  \right\}$ the acceptance set and  by $ R_s(x)= A_s(x)^c$ its complementary set. Then, we recall $V_s(x)= c_s\lstat_s(x)^{-\beta}$ for some $\beta\in]0,1[$. Therefore, for all $x\in \Xspace$:
\begin{multline*}
	\frac{\ntrans_s  V_s(x)}{V_s(x)} = 
	\int_{A_s(x)} \cand(x,z)\frac{V_s(z)}{V_s(x)} dz +\\
	\int_{R_s(x)} \frac{\lstat_s(z)\cand(z,x)}{\lstat_s(x)\cand(x,z)}\cand(x,z)\frac{V_s(z)}{V_s(x)} dz + \\
	\int_{R_s(x)}  \left( 1-\frac{\lstat_s(z)\cand(z,x)}{\lstat_s(x)\cand(x,z)}\right)\cand(x,z)dz\\
	\leq\int_{A_s(x)} \underbrace{\frac{\lstat_s(z)^{-\beta}}{\lstat_s(x)^{-\beta}}\cand(x,z)}_{f_1(x,z)} dz +\\
	\int_{R_s(x)} \underbrace{\frac{\lstat_s(z)^{1-\beta}}{\lstat_s(x)^{1-\beta}}\cand(z,x)}_{f_2(x,z)} dz +
	\int_{R_s(x)} \underbrace{\cand(x,z) }_{f_3(x,z)}dz\,.
\end{multline*}

On the acceptance set $A_s(x)$, we have:
\begin{equation*}
	\frac{\lstat_s(z)^{-\beta}}{\lstat_s(x)^{-\beta}}\cand(x,z) \leq \cand(z,x)^\beta \cand(x,z)^{1-\beta}\,.
\end{equation*}
Thanks to Equation \eqref{eq:encadrementCand} one can bound this right hand side by the following symmetric Gaussian distribution:
\begin{equation}
	\frac{\lstat_s(z)^{-\beta}}{\lstat_s(x)^{-\beta}}\cand(x,z) \leq k_2 g_{\epsilon_2}(z-x)
\end{equation}
which yields:
\begin{equation}
	\label{eq:dominationsurA}
	\int_{A_s(x)} f_1(x,y) dz \leq k_2 \int_{A_s(x)} 	g_{\epsilon_2}(z-x) dz\,.
\end{equation}
Equivalently on $R_s(x)$, we have the following bound:
\begin{eqnarray}
	\label{eq:dominationsurR}
	\frac{\lstat_s(z)^{1-\beta}}{{\lstat_s(x)^{1-\beta}}}\cand(z,x) &\leq &\cand(x,z)^{1-\beta}\cand(z,x)^{\beta} \\
	&\leq & k_2 g_{\epsilon_2}(z-x)\,.
\end{eqnarray}
Let fix $\eps>0$, there exists $ a>0$ such that $\int_{B(x,a)} g_{\epsilon_2} (z-x)dz \geq 1-\eps$. This leads to:
\begin{equation*}
	\int_{A_s(x)\cap B(x,a)^c} f_1(x,z) dz \leq  k_2 \eps \,.
\end{equation*}

Let $C_{\lstat_s(x)}$ be the level set of $\lstat_s$ in $x$: $C_{\lstat_s(x)}= \{z\in\Xspace \ : \ \ \lstat_s(z)=\lstat_s(x)\}$. We define a pipe around this level set as 

\noindent$C_{\lstat_s(x)}(u)=\{ z+t \ n(z), \ |t|\leq u, \ z\in C_{\lstat_s(x)} \}$.

Thanks to assumption (\textbf{B1}), there exists $r_1>0$ such that for all $ x\in\Xspace$ satisfying $|x| \geq r_1$ then $0$ is inside the hyperspace defined by the level set $C_{\lstat_s(x)}$ ($\lstat_s(0) > \lstat_s(x)$). Therefore, let $x\in\Xspace, \ |x|\geq r_1$, then for all $z\in \Xspace$, $\exists x_1\in C_{\lstat_s(x)}$ and $t>0$ such that $z=x_1+t\ n(x_1)$. 

Since $z\mapsto g_{\epsilon_2}(z-x)$ is a smooth density in the variable $z$, we can find $u>0$ sufficiently small such that 
\begin{equation}
	\label{eq:intsurCpiinterB}
	\int_{B(x,a)\cap C_{\lstat_s(x)}(u)} g_{\epsilon_2}(z-x)dz \leq \eps\,,
\end{equation}
leading to
\begin{eqnarray*}
	\int_{A_s(x)\cap B(x,a) \cap C_{\lstat_s(x)}(u)} f_1(x,z) dz &\leq & k_2 \eps \,.
\end{eqnarray*}
Assumption (\textbf{B1}) implies that for any $r>0$ and $t>0$, $d_r(t)=\sup\limits_{|x|\geq r}\frac{\lstat_s(x+t\ n(x))}{\lstat_s(x)}$ goes to $0$ as $r$ goes to $\infty$. 
Denote $C_{\lstat_s(x)}(u)^{c+} = \{ z \in C_{\lstat_s(x)}(u)^c \  s.t. \  \lstat_s(x)>\lstat_s(z) \}$ and $C_{\lstat_s(x)}(u)^{c-} = \{ z \in C_{\lstat_s(x)}(u)^c \  s.t.\ \lstat_s(x)<\lstat_s(z) \}$. We denote $ \mathcal{D}^+ =A_s(x)\cap B(x,a) \cap C_{\lstat_s(x)}(u)^{c+} $.
 Therefore there exists $r_2>r_1+a$ such that for any $x, \ |x|\geq r_2$  
\begin{eqnarray*}
	\int_{\mathcal{D^+}} f_1(x,z) dz &\leq   &
	\int_{\mathcal{D^+}} \left(\frac{\lstat_s(z)}{\lstat_s(x)}\right)^{1-\beta} \cand(z,x)dz \\
	&\leq&  d_{r_2}(u)^{1-\beta} k_2 \int_\Xspace g_{\epsilon_2}(z-x) dz \\
	 &\leq & k_2 d_{r_2}(u)^{1-\beta}\,,
\end{eqnarray*}
using Equation \eqref{eq:superexp2} which states that the stationary distribution is decreasing in the direction of the normal of $x$ sufficiently large.

In the same way, one has on the set $ \mathcal{D}^- =A_s(x)\cap B(x,a) \cap C_{\lstat_s(x)}(u)^{c-} $
\begin{eqnarray*}
	\int_{\mathcal{D}^-} f_1(x,z) dz &\leq& \int_{\mathcal{D}^-}  \left(\frac{\lstat_s(z)}{\lstat_s(x)}\right)^{-\beta} \cand(x,z)dz\\
	&\leq & k_2 d_{r_2}(u)^{\beta}\,.
\end{eqnarray*}
The same inequalities can be obtained for $f_2$ using the same arguments:
\begin{eqnarray*}
	\int_{R_s(x)\cap B(x,a)^c} f_2(x,z) dz &\leq & k_2 \eps \\
	\int_{R_s(x)\cap B(x,a) \cap C_{\lstat_s(x)}(u)} f_2(x,z) dz &\leq & k_2 \eps \\
	\int_{R_s(x)\cap B(x,a) \cap C_{\lstat_s(x)}(u)^{c+}} f_2(x,z) dz &\leq & k_2 d_{r_2}(u)^{1-\beta}\\
	\int_{R_s(x)\cap B(x,a) \cap C_{\lstat_s(x)}(u)^{c-}} f_2(x,z) dz &\leq & k_2 d_{r_2}(u)^{\beta}\,.
\end{eqnarray*}

This yields
\begin{equation}
	\limsup\limits_{|x|\to \infty} \frac{\ntrans_s V_s (x)}{V_s(x)} \leq \limsup\limits_{|x|\to \infty} \int_{R_s(x)} \cand(x,z)dz.
\end{equation}
Let $Q(x,A_s(x)) = \int_{A_s(x)} \cand(x,z)dz$, we get $$\limsup\limits_{|x|\to \infty} \frac{\ntrans_s V_s (x)}{V_s(x)} \leq 1-\liminf\limits_{|x|\to \infty} Q(x,A_s(x)).$$  
Let us now prove that 

\noindent $\liminf\limits_{|x|\to \infty} Q(x,A_s(x)) \geq~c>0$ where $c$ does not depend on $x$.

Let $a$ fixed as above. Since $\cand$ is an exponential function, there exists $c_0^a>0$ such that for all $x\in\Xspace$ and $s\in\Sr$,
\begin{equation}
	\label{eq:candpositif}
	\inf\limits_{z\in B(x,a)}\frac{\cand(z,x)}{\cand(x,z)}\geq c_0^a\,.
\end{equation}
 Moreover, thanks to assumption (\textbf{B1}) there exists $r_3>0$ such that for all $x\in \Xspace$, $|x|\geq r_3$, there exists $0<u_2<a$ such that, 
\begin{equation}
	\label{eq:limitpi}
	\frac{\lstat_s(x)}{\lstat_s(x-u_2\ n(x))}\leq c_0^a\,.
\end{equation} 
Hence, for $|x|\geq r_3$, any point $x_2=x-u_2\ n(x)$ belongs to $A_s(x)$.

Let $W(x)$ be the cone defined as: 
\begin{multline}
	W(x) = \left\{  x_2-t\zeta, \ 0<t<a-u_2, \ \zeta\in \Sr^{d-1}, \right.\\ \left. |\zeta - n(x_2)|\leq \frac{\eps}{2}  \right\}
	\end{multline}
where $\Sr^{d-1}$ is the unit sphere in $\R^d$.
 \\

Let us prove that $W(x)\subset A_s(x)$.

Using assumption (\textbf{B1}), we have for a sufficiently large $x$: $m(x).n(x)\leq -\eps$. Besides, by construction of $W(x)$ for large $x$, for all $z\in W(x)$, $|n(z)-n(x)|\leq \eps/2$ with $n(x)=n(x_2)$ (see Fig.~\ref{fig:pics_dessin}). This leads to for any sufficiently large $x $, for all $ z\in W(x)$, 
\begin{multline}
	\label{eq:mzeta}
	m(z).\zeta =  m(z). (\zeta-n(x_2)) +  m(z).(n(x_2)-n(z)) \\
	+ m(z). n(z) \leq \eps/2 + \eps/2 - \eps =0\,.
\end{multline}
Let now $z=x_2-t\zeta \in W(x)$. Using the mean value theorem on the differentiable function $\lstat_s$ between $x_2$ and $z$, we get that there exists $\tau \in ]0,s[$ such that $\lstat_s(z)-\lstat(x_2)= -t\zeta . \nabla \lstat_s(x_2-\tau\zeta)$. Using the definition of $m$, this implies that
$\lstat_s(z)-\lstat_s(x_2)= -t\zeta .m(x_2-\tau\zeta)  |\nabla \lstat_s(x_2-\tau\zeta)| \geq 0$ thanks to Equation \eqref{eq:mzeta}. Putting all these results together we finally get that for all $z\in W(x)$, $\lstat_s(z)\geq\lstat_s(x_2)\geq \frac{1}{c_0^a}\lstat_s(x)$. Moreover, as $z\in B(x,a)$ as well, Equation \eqref{eq:candpositif} is satisfied, leading to $z\in A_s(x)$. \\

Then, we have
\begin{eqnarray*}
Q(x,A_s(x))& = &\int_{A_s(x)} \cand(x,z)dz \\
& \geq &\int_{A_s(x)} k_1 g_{\epsilon_1}(z-x) dz \\
& \geq &k_1 \int_{W(x)}g_{\epsilon_1}(z-x) dz  \\
& = &\int_{T_x(W(x))}g_{\epsilon_1}(z) dz 
\end{eqnarray*}
where 
\begin{multline}
T_x(W(x))= \Bigl\{  -u_2 \ n(x)-t\zeta, \ 0<t<a-u_2, \\ \Bigr.  \left. \zeta\in \Sr^{d-1}, \ |\zeta - n(x)|\leq \frac{\eps}{2}  \right\}
\end{multline}
 is the translation of the set $W(x)$ by the vector $x$. Note that $W(x)$ does not depend on $s$. But since $g_{\epsilon_1}$ is isotropic and $T_x(W(x))$ only depends on a fixed constant $u_2$ and $n(x)$, this last integral is independent of $x$, so there exists a positive constant $c$ independent of $s\in\Sr$ such that:
\begin{equation}
\label{eq:defdec}
c=\int_{T_x(W(x))}g_{\epsilon_1}(z) dz \,.
\end{equation} 

Back to our limit, for all $s\in\Sr$
\begin{equation}
	\limsup\limits_{|x|\to \infty} \frac{\ntrans_s V_s (x)}{V_s(x)} \leq 1-c	
\end{equation} 
which ends the proof of the condition \eqref{eq:limsupinf1}. 

To prove \eqref{eq:supfini}, we use the previous result. Indeed, since $\frac{\ntrans_s V_s (x)}{V_s(x)} $ is a smooth function on $\Xspace$ it is bounded on every compact subset. Moreover since the $\limsup$ is finite, then it is also bounded outside a fixed compact. This proves the results.\\

Thanks to assumption \textbf{(B2)} and the bounded drift for all $s\in\Sr$, there exists a constant $c_0^a$ uniform in $\ass\in\Sr$  such that Equations \eqref{eq:candpositif}  and \eqref{eq:limitpi} still hold for all $\ass \in \Sr$. This implies, as mentioned above, that the set $T_x(W(x))$ is independent of $\ass\in\Sr$. Therefore, we can set $\tilde{\lambda} = 1-c <1$ where $c$ is defined in Equation \eqref{eq:defdec} and is also independent of $\ass\in\Sr$. \\

This proves the Drift property for the function $V_s$: there exist constants $0<\tilde\lambda<1 $ and $\tilde{b}>0$ such that for all $x\in\Xspace$,
\begin{equation}
	\ntrans_s V_s(x)\leq \tilde\lambda V_s (x) + \tilde{b}\mathds{1}_{\smallset} (x)\,,
\end{equation}
where $\smallset$ is a small set. Note that $\tilde{b}$ is also independent of $s\in\Sr$ using the same arguments as before.

Let us now exhibit a function $V$ independent of $s\in \Sr$ and prove the uniform Drift condition.\\

We define for all $x\in\Xspace$, 
\begin{equation}
	\label{eq:defV}
	V(x)= V_1(x)^\xi V_2(x)^{2\xi}
\end{equation}
for  $0<\xi<\min(1/2\beta, b_0/4)$. Therefore, for all $s\in\Sr$, for all $\varepsilon >0$ we have,
\begin{multline}
	\ntrans_s V(x) = \int_\Xspace \ntrans_s (x,z) V_1(z)^\xi V_2(z)^{2\xi} dz\\
	\leq \frac{1}{2}\int_\Xspace \ntrans_s (x,z) \left(\frac{V_1(z)^{2\xi}}{\varepsilon ^2} + \varepsilon^2 V_2(z)^{4\xi}\right) dz\\
	\leq \frac{1}{2\varepsilon^2} \int_\Xspace \ntrans_s (x,z) V_s(z)^{2\xi} dz + \\ \frac{\varepsilon^2}{2} \int_\Xspace \ntrans_s (x,z) V_2(z)^{4\xi} dz \,.
\end{multline}
Applying the Drift property for $\ntrans_s$ with $V_s^{2\xi}$, 
\begin{multline}
	\ntrans_s V(x) \leq \frac{1}{2\varepsilon^2} (\tilde\lambda V_s(x)^{2\xi} +  \tilde{b}\mathds{1}_\smallset(x)) \\+  \frac{\varepsilon^2}{2} \int_\Xspace \ntrans_s (x,z) V_2(z)^{4\xi} dz\,.
\end{multline}
Using the definition of $V$ and the fact that $V_1$ is bounded by below by $1$, we get:
\begin{multline}
	\ntrans_s V(x) 
\leq \frac{\tilde\lambda}{2\varepsilon^2} V(x) +  \frac{\tilde{b}}{2\varepsilon^2}\mathds{1}_\smallset(x) \\ 
+ \frac{\varepsilon^2}{2} \int_\Xspace \ntrans_s (x,z) V_2(z)^{4\xi} dz\,.
\end{multline}

Since $0<\tilde\lambda<1$ is independent of $s\in\Sr$ and using assumption \textbf{(B3)}, there exists $\xi >0$ such that 
\begin{equation}
	 \sup\limits_{s\in \Sr, x\in\Xspace}\int_\Xspace \ntrans_s (x,z) V_2(z)^{4\xi} dz \leq \frac{2}{1+\tilde\lambda}\,.
\end{equation}
This yields
\begin{multline}
		\ntrans_s V(x) 
	\leq \left(\frac{\tilde\lambda}{2\varepsilon^2} + \frac{\varepsilon^2}{1+\tilde\lambda}\right) V(x) + \frac{\tilde{b}}{2\varepsilon^2}\mathds{1}_\smallset(x)\,.
\end{multline}
We can now fix $\varepsilon ^2 = \sqrt{\frac{\tilde\lambda(1+\tilde\lambda)}{2}}$ which leads to
\begin{equation}
	\ntrans_s V(x) 
\leq \sqrt{\frac{2\tilde\lambda}{1+\tilde\lambda}} V(x) + \frac{\tilde{b}}{2\varepsilon^2}\mathds{1}_\smallset(x)\,.
\end{equation}

We set $\lambda = \sqrt{\frac{2\tilde\lambda}{1+\tilde\lambda}} <1$ and $b=\frac{\tilde{b}}{2\varepsilon^2}>0$ which concludes the proof.


\subsection{Proof of Theorem \ref{Th:convergenceAlgo2} }
\label{subsec:Convergence}

We provide here the proof of the convergence of the estimated sequence generated by Algorithm \ref{algo:AMALASAEM}.

We apply Theorem 4.1 from \cite{aktdefmod} with the functions $H_s$ equals to $H_s(\hid)=\tS(z)-s$, $\ntrans_s = \ntrans_{\hte(s)}$, $\pi_s=p_{\hte(s)}$  and 
	\begin{eqnarray*}
	   h(s) = \int (S(\hid)-s) p_{\hte(s)}(\hid)\mu(d\hid) \, .
	\end{eqnarray*}

Let us first prove assumption {\bf (A1')} which ensures the existence of a global Lyapunov function for the mean field of the stochastic approximation. It guaranties that, under some conditions, the sequence  $(s_k)_{k\geq 0}$ remains in a compact subset of $\Sr$ and converges to the set of critical points of the log-likelihood.

Assumptions {\bf (M1)-(M7)} ensure that $\Sr$ is an open subset and that the function $h$ is 
continuous on $\Sr$. Moreover defining $w(s)=-l(\hte(s))$, we get that $w$ is continuously differentiable on $\Sr$.
Applying Lemma 2 of \cite{DLM}, we get {\bf (A1')(i)}, {\bf (A1')(iii)} and {\bf (A1')(iv)}.

To prove {\bf (A1')(ii)}, we consider as absorbing set $\Sr_a$ the closure of the convex hull of $S(\Rset{\dhid})$ denoted $\overline{Conv(S(\Rset{\dhid}))}$. So assumption {\bf (M7)(ii)} is exactly equivalent to assumption {\bf (A1')(ii)}.

This achieves the proof of assumption {\bf (A1')}.

\vspace{0.5cm}

Let us now prove assumption  \textbf{(A2)} which states in particular the existence of a unique invariant distribution for the Markov chain. 

To that purpose, we prove that our family of kernels satisfies the drift conditions mentioned in \cite{andrieumoulinespriouret} and used in \cite{aktdefmod} in a similar context. These conditions are the existence of a small set uniformly in $s\in\Kapa$, the uniform drift condition and an upper bound on the family kernel~:
\begin{itemize}
\item[(DRI1)]
For any $s\in \Sr$, $\ntrans_{\hte(s)}$ is $\psi$-irreducible and aperiodic. In addition there exist  a function $V:\R^{\dhid} \to [1,\infty[$  and a constant $p\geq 2$ such that for any compact subset $\Kapa\subset \mathcal{S}$,  there exist an integer $j$ and constants $0<\lambda<1$, $B$, $\kappa$,  $\delta>0$ and a probability measure $\nu$ such that
\begin{eqnarray}
\label{6.1}\sup\limits_{s\in\Kapa} \ntrans_{\hte(s)} ^jV^p(\hid)& \leq &\lambda V^p(\hid) +
 B \mathds{1}_{\texttt{C}}(\hid) \,, \\ \label{6.2}
 \sup\limits_{s\in\Kapa} \ntrans_{\hte(s)} V^p(\hid)& \leq& \kappa V^p(\hid)
 \ \  \forall \hid\in \Xspace  \,, \\ \label{6.3}
 \inf\limits_{s\in\Kapa} \ntrans_{\hte(s)} ^j (\hid,A) &\geq& \delta \nu(A) \ \
 \forall \hid\in \texttt{C}, \forall A \in \mathcal{B} \, .
 \end{eqnarray}
\end{itemize}

Let us start with the irreducibility of $\ntrans_{\hte(s)}$.
The kernel $\ntrans_{\hte(s)}$ is bounded by below as follows~:
\begin{equation}
	\ntrans_{\hte(s)}(x,A) \geq \int_A \alpha_s(x,z) \cand(x,z)dz \,,
\end{equation}
where $\alpha_s(x,z) =\min (1,\rho_s(x,z))  $ and $\rho_s(x,z)= \frac{\lstat_s(z)\cand(z,x)}{\cand(x,z)\lstat_s(x)} >0$. Since the proposal density $\cand$ is positive, this proves that  $\ntrans_s(x,A)$ is positive and the $\psi$-irreducibility of each kernel of the family.\\

Proposition~\ref{ergogeo} and Remark~\ref{rem:ergoPuissancep} show that Equations~\eqref{6.1} and \eqref{6.3} hold for $j=1$ with $V$ defined in Equation~\eqref{eq:defV} and some $p>2$. Moreover, since Equation~\eqref{6.1} holds for $j=1$ and $V\geq 1$, Equation~\eqref{6.2} directly comes from Equation~\eqref{6.1} choosing $\kappa= B+\lambda$. This implies all three inequalities.  Since the small set condition is satisfied with $j=1$ (small set "in one-step"), each chain of the family is aperiodic (see \cite{meyntweedie}). 

Assumption {\bf (A2)} is therefore directly implied by assumption {\bf (M1)}. 

\vspace{0.5cm}

Let us now prove assumption {\bf (A3')} which states some regularity conditions (H\"older type ones) on the solution of the Poisson equation related to the transition kernel. It also ensures that this solution and its image through the transition kernel have reasonable behaviors as the chain goes to infinity and that the kernel is $V^p$-bounded in expectation.\\

The drift conditions proved previously imply the geometric ergodicity uniformly in $s$ in any compact set $\Kapa$. This also ensures the existence of a solution of the Poisson equation (see \cite{meyntweedie}) required in Assumption (\textbf{A3'}).\\

We first consider condition (\textbf{A3'(i)}). 

Let us define for any $g :\Xspace \to \R^{m} $ the norm $\|g\|_V \triangleq \sup\limits_{\hid\in\Xspace} \frac{\|g(\hid)\|}{V(\hid)}$.\\

\noindent Since $H_s(\hid)=\tS(\hid) -s $, assumptions {\bf (M8)} and (\textbf{B1}) ensure that 
$ \sup\limits_{s\in\Kapa} \|H_s\|_V < \infty $ 
 and  inequality (4.3) of (\textbf{A3'(i)}) holds. \\
The uniform
ergodicity of the family of Markov chains corresponding to the AMALA on $\Kapa$ ensures that there
 exist constants $0<\gamma_\Kapa<1$ and  $C_\Kapa>0$  such that for all $s\in \Kapa$
\begin{eqnarray*}
 \sup\limits_{s\in\Kapa} \|g_{\hat\te(s)}\|_V &=& \sup\limits_{s\in\Kapa} \| \sum\limits_{k\geq 0} (\ntrans^k_{\hat{\te}(s)}
  H_s - p_{\hte(s)}H_s)\|_V \\
&\leq & \sup\limits_{s\in\Kapa} \sum\limits_{k\geq 0} C_\Kapa \gamma_\Kapa^k \|H_s\|_V <\infty \ .
\end{eqnarray*}
Thus for all $ s$ in $ \Kapa$ , $\ g_{\hat\te(s)}$ belongs to $\mathcal{L}_V = \{ g~ :
\R^{\dhid} \to \R^m
, \| g\| _V < \infty \} $.

Repeating the same calculation as above, it is immediate that
$ \sup\limits_{s\in\Kapa}|\|\ntrans_{\hat{\te}(s)}
g_{\hat\te(s)} \|_V$ is bounded. This ends the proof of inequality (4.4) of
(\textbf{A3'(i)}).\\

We now move to the H\"older  conditions (4.5) of (\textbf{A3'(i)}). We will use
the two following lemmas which state H\"older conditions on the
transition kernel and its iterates:
\begin{lemma}\label{lem:Holder}
Let $\mathcal{K}$ be a compact subset of $\Sr$. There exists a
constant $C_{\mathcal{K}}$ such that
for all $1\leq p $ there exists $ q>p$, for all function $f \in \mathcal{L}_{V^p}$ and for all $(s,s')
\in \mathcal{K}^2$ we have~:
\begin{eqnarray*}
  \| \ntrans_{\hat\te(s)} f -
  \ntrans_{\hat\te(s')} f \|_{V^{q}}  \leq
C_\mathcal{K}  
\|f\|_{V^{p}} \  \|s-s'\| \,.
\end{eqnarray*}
\end{lemma}


\begin{proof}
 For any $f\in \mathcal{L}_{V^p}$ and any $x\in \R^l$, we have
\begin{multline*}
\ntrans_{s}f(x) = \int_{\R^\dhid} f(\hid)\alpha_s(x,\hid) \cand(x,\hid)d\hid \\+ f(x)  (1-\alpha_s(x))\,
\,,
\label{eq:9}
\end{multline*}
where $\alpha_s(x,\hid)=\min\left(1,\frac{p_{\hte(s)}(\hid)\cand(\hid,x)}{\cand(x,\hid)p_{\hte(s)}(x)}\right)$ and 

\noindent $\alpha_s(x)=\int \alpha_s(x,\hid)\cand(x,\hid)d\hid$ is the average acceptance rate. Let us denote for all $x$, $z$ and $s$:
 $ r_s(x,\hid) = \alpha_s(x,\hid)\cand(x,\hid)$.

Let $s$ and $s'$ be two points in $\mathcal{K}$. We note that $s\mapsto \hat{\te}(s)$ is a  continuously differentiable  function therefore uniformly bounded in $s\in\Kapa$.

\begin{multline*}
  \| \ntrans_{s} f (x) -
  \ntrans_{s'} f (x) \|  \leq \|f\|_{V^{p}} \times \\ \left\{   
\int_\Xspace |r_s(x,\hid) - r_{s'}(x,\hid)|V^p(\hid) d\hid \right.+  \\ \left. V^p(x) \int_\Xspace |r_s(x,\hid) - r_{s'}(x,\hid)| d\hid
\right\}  \,,\\
\leq 2 \|f\|_{V^{p}} V^p(x)    \times \\
\int_\Xspace |r_s(x,\hid) - r_{s'}(x,\hid)|V^p(\hid) d\hid \,.
\end{multline*}

Let $I = \int_\Xspace |r_{s}(x,\hid) - r_{s'}(x,\hid)|V^p(\hid) d\hid$.
For sake of simplicity, we denote by $A_s$ the acceptance set instead of $A_s(x)$. We decompose $I$ into four terms~:
\begin{eqnarray}
	I &=\int_{A_s\cap A_{s'}} |r_{s}(x,\hid) - r_{s'}(x,\hid)|V^p(\hid) d\hid \nonumber
	\\
	\label{I1}\\
	+& \int_{A_s\cap A_{s'}^c} |r_{s}(x,\hid)- r_{s'}(x,\hid)|V^p(\hid) d\hid \nonumber
	\\
	 \label{I2}\\
	+& \int_{A_s^c\cap A_{s'}} |r_{s}(x,\hid) - r_{s'}(x,\hid)|V^p(\hid) d\hid \nonumber
	\\
	\label{I3}\\
	+& \int_{A_s^c\cap A_{s'}^c} |r_{s}(x,\hid) - r_{s'}(x,\hid)|V^p(\hid) d\hid \nonumber
	\\
	\label{I4}\,.
\end{eqnarray}

Let us first consider the term~\eqref{I1}. 
\begin{multline}
	\int_{A_s\cap A_{s'}} |r_{s}(x,\hid) - r_{s'}(x,\hid)|V^p(\hid) d\hid \\ = \int_{A_s\cap A_{s'}} |\cand(x,\hid) - q_{s'}(x,\hid)|V^p(\hid) d\hid \,.
\end{multline}
We use the mean value theorem on the smooth function $s\mapsto q_s(x,z)$ for fixed values of $(x,z)$.

\begin{eqnarray*}
	\frac{d q_s(x,z)}{ds} = q_s(x,z) \frac{d \log q_s(x,z)}{ds}\,.
\end{eqnarray*}
After some calculations, using the bounded drift and covariance and Assumption \textbf{(M8)}, we get~:
 \begin{eqnarray*}
	\frac{d \log q_s(x,z)}{ds} &\leq&  \tilde{P_1}(x,z)  \left( \left\| \frac{d D_s(x)}{ds}\right\| + \left\| \frac{d \Sigma_s(x)}{ds}\right\|_F \right. \\  &+& \left. \left\| \frac{d \Sigma^{-1}_s(x)}{ds}\right\|_F  \right) \,,\\
	& \leq & C_\Kapa P_1(x,z)\,.
\end{eqnarray*}
where $D_s$ and $\Sigma_s$ are respectively the drift and covariance of the proposal $q_s$ and $\tilde{P_1}$ and $P_1$ are two polynomial functions in both variables.

Using Equation~\eqref{eq:encadrementCand}, we have~:
\begin{eqnarray*}
	\left|\frac{d q_s(x,z)}{ds}\right| \leq k_2  C_\Kapa P_1(x,z) g_{\epsilon_2}(z-x) \,,
\end{eqnarray*}
which leads to~:
\begin{multline*}
	\int_{A_s\cap A_{s'}} |\cand(x,\hid) - q_{s'}(x,\hid)|V^p(\hid) d\hid  \\ \leq  k_2  C_\Kapa \|s-s'\| \int_{\Xspace} V^p(z) P_1(x,z) g_{\epsilon_2}(z-x) dz \\
	\leq  k_2 C_\Kapa Q_1(x) \|s-s'\|\,,
\end{multline*}
where $Q_1$ is a polynomial function.\\

Now we move to the second term \eqref{I2}.
Let $z\in A_s \cap A_{s'}^c$. 
We define for all $u \in [0,1]$ the barycenter $s(u)$ of $s$ and $s'$ equals to $us + (1-u)s'$ which belongs to 
the convex hull of the compact subset $\Kapa$.

Since $u\mapsto \rho_{s(u)}(x,z)$ is continuously differentiable, $\rho_s(x,z)\geq 1$ and $\rho_{s'}(x,z)<1$, using the intermediate value theorem, there exists $u \in ]0,1]$ depending on $x$ and $z$ such that $\rho_{s(u)}(x,z) = 1$. We choose the minimum value of $u$ satisfying this condition.  Therefore, 
\begin{multline}
	\label{eq:intermediaire}
	|\alpha_s(x,\hid)\cand(x,\hid) - \alpha_{s'}(x,\hid)q_{s'}(x,\hid)| \\ \leq |\cand(x,\hid) - q_{s(u)}(x,z)| +\\ |\rho_{s(u)}(x,\hid)q_{s(u)}(x,z) - \rho_{s'}(x,\hid)q_{s'}(x,\hid)|\,.
\end{multline}

We treat the first term of the right hand side as previously. For the second term, we use the mean value theorem for the function  $v\mapsto f_{s(v)}(x,z)= \rho_{s(v)}(x,z)q_{s(v)}(x,z)$ on $ ]0, u[$. There exists $v\in ]0, u[$ such that
\begin{equation*}
	|f_{s(u)}(x,z) - f_{s'}(x,\hid)| \leq \left|\frac{ d f_{s(v)}(x,z)}{dv}\right| \|s - s'\|\,.
\end{equation*}
Thanks to the upper bound above we get
\begin{eqnarray*}
	\frac{d \log f_{s(v)}(x,z)}{dv} &\leq&  C_\Kapa P_2(x,z)\,,
\end{eqnarray*}
where $P_2$ is a polynomial function in both variables. Since  on the segment defined by $s(u)$ and $s'$ we have $\rho_s(x,z)\leq 1$~:
\begin{eqnarray*}
	\frac{d f_{s(v)}(x,z)}{dv} &= &f_{s(v)}(x,z) \frac{d \log f_{s(v)}(x,z)}{dv}\\
	&\leq & C_\Kapa q_{s(v)}(x,z) P_2(x,z)\\
	&\leq & k_2 C_\Kapa \|s-s'\| P_2(x,z) g_{\epsilon_2}(z-x)\,.
\end{eqnarray*}
This yields~:
\begin{multline*}
	\int_{A_s\cap A_{s'}^c} |r_{s}(x,\hid) - r_{s'}(x,\hid)|V^p(\hid) d\hid \\ \leq  k_2 C_\Kapa \Bigl(Q_1(x) + Q_2(x)\Bigr)\|s-s'\|\,. 
\end{multline*}
\\

The third term~\eqref{I3} is the symmetric one of the second.

Let us end with the last term \eqref{I4}. 
\begin{eqnarray*}
	\int_{A_s^c\cap A_{s'}^c} |\alpha_s(x,\hid)\cand(x,\hid) - \alpha_{s'}(x,\hid)q_{s'}(x,\hid)|\times \\
	V^p(\hid) d\hid \\ = \int_{A_s^c\cap A_{s'}^c} |\rho_s(x,\hid)\cand(x,\hid) - \rho_{s'}(x,\hid)q_{s'}(x,\hid)| \times \\V^p(\hid) d\hid\,.
\end{eqnarray*}
If for all $u\in ]0,1[$, $\rho_{s(u)}(x,y) < 1$ then this term can be treated as the second term of Equation~\eqref{eq:intermediaire}.
If there exists $u\in ]0,1[$ such that $\rho_{s(u)}(x,y) \geq 1$, we define $u_0$ and $u_1$ respectively the smallest and biggest elements in $]0,1[$ such that $\rho_{s(u_0)}=\rho_{s(u_1)}=1$. The first and last terms are treated as the previous case and the middle term is treated as the term~\eqref{I1}.
\\

Putting all these upper bounds together yields~:
\begin{equation}
	\| \ntrans_{s} f (x) -
  \ntrans_{s'} f (x) \|  \leq  2\|f\|_{V^{p}}V^p(x) Q(x) \|s-s'\|\,,
\end{equation}
where $Q$ is a polynomial function in $x\in\Xspace$. Therefore, there exists a constant $q>p$ such that $V^p(x)Q(x)\leq V^q(x)$ which concludes the proof.

\end{proof}

\hspace{0.5cm}

\begin{lemma}\label{lem67}
Let $\mathcal{K}$ be a compact subset of $\Sr$. There exists a
constant $C_\mathcal{K}$ such that
for  all $1\leq p < q $, for all function $f \in \mathcal{L}_{V^p}$, for all $(s,s')
\in \mathcal{K}^2$ and for all $k\geq 0$, we have:
  \begin{eqnarray*}
\| \ntrans_{\hat{\te}(s)} ^k f -
  \ntrans_{\hat{\te}(s')} ^k f \|_{V^{q}}  \leq
C_\mathcal{K}  
\|f\|_{V^{p}} \|s-s'\| \ .
\end{eqnarray*}
\end{lemma}

\begin{proof}
The proof of lemma \ref{lem67} follows the line of the proof of Proposition B.2 of~\cite{andrieumoulinespriouret}.\\
\end{proof}

\bigskip

Thanks to the proofs of \cite{aktdefmod}, we get that $h$ is a H\"older function for any $0<a<1$
which leads to (\textbf{A3''(i)}).\\

We finally focus on the proof of (\textbf{A3''(ii)}).

\begin{lemma}
\label{lem:3}
  Let $\mathcal{K}$ be a compact subset of $\Sr$ and $p\geq 1$.
 For all sequences  $\boldsymbol\gamma=(\gamma_k)_{k\geq 0}$ and 
  $\boldsymbol\varepsilon=(\varepsilon_k)_{k\geq 0}$ satisfying $\varepsilon_k< \overline\varepsilon$ for some $\overline\varepsilon$ sufficiently small, there exists  $C_{\mathcal{K}}>0$,
   such that  
 for any $\hid_0\in\Xspace$, we have
$$\sup_{s\in \mathcal{K}}\sup_{k\geq
  0}\mathbb{E}_{\hid,s}^{\boldsymbol\gamma}
[V^p(\hid_k)\mathds{1}_{\sigma(\mathcal{K})\land \nu(\boldsymbol\varepsilon)\geq
  k}]\leq C_{\mathcal{K}} V^{p}(\hid_0)\,,$$
where $\mathbb{E}_{\hid,s}^{\boldsymbol\gamma}$ is the expectation related to the non-homogeneous Markov chain $((\hid_k,s_k))$ started from $(\hid,s)$ with step size sequence ${\boldsymbol\gamma}$.
\end{lemma}

\begin{proof}
  Let $K$ be a compact subset of $\Theta$ such that
  $\hat\te(\mathcal{K})\subset K$. We note in the sequel,
  $\te_k=\hat\te(s_k)$. We have for $k\geq 2$, using the Markov property and
  the drift property \eqref{eq:drift} for $V^p$, 
\begin{eqnarray}
  \mathbb{E}_{\hid,s}^{\boldsymbol\gamma} [V^{p}(\hid_k)\mathds{1}_{\sigma(\mathcal{K})\land
    \nu(\boldsymbol\varepsilon)\geq k}]   
	&\leq& 
	\mathbb{E}_{\hid,s}^{\boldsymbol\gamma}
  [\ntrans_{\te_{k-1}}V^{ p}(\hid_{k-1})] \nonumber \\
	\label{eq:majorationDrift}
\\
&\leq &
\lambda \mathbb{E}_{\hid,s}^{\boldsymbol\gamma}
[V^{ p}(\hid_{k-1})]+ C \,.\nonumber\\
\label{eq:majorationDrift}
\end{eqnarray}

Iterating the same arguments recursively leads to~: 
\begin{multline*}
	\mathbb{E}_{\hid,s}^{\boldsymbol\gamma} [V^{p}(\hid_{k})\mathds{1}_{\sigma(\mathcal{K})\land
	    \nu(\boldsymbol\varepsilon)\geq k}] 
	\\ \leq
	 \lambda^k \mathbb{E}_{\hid,s}^{\boldsymbol\gamma} [V^{p}(\hid_0)] + 	C \sum\limits_{l=0}^{k-1} \lambda^l \,.
\end{multline*}

Since $\lambda<1$ and $V(\hid)\geq 1$ for all $\hid \in \Xspace$, 
 for all $k\in\N$, we have~:
\begin{eqnarray*}
 \mathbb{E}_{\hid,s}^{\boldsymbol\gamma} [V^{p}(\hid_k)\mathds{1}_{\sigma(\mathcal{K})\land
    \nu(\boldsymbol\varepsilon)\geq k}] &\leq&
 V^{p}(\hid_0) + C\sum\limits_{l=0}^{k-1}\lambda^l\\
& \leq & V^{p}(\hid_0) \left(1+ \frac{C}{1-\lambda}\right)\,.
\end{eqnarray*}

\end{proof}

This yields (\textbf{A3'(ii)}) which concludes the proof of Theorem~\ref{Th:convergenceAlgo2}.\\

\subsection{Proof of the Central Limit Theorem for the Estimated Sequence}
\label{appendix:TCL}

The proof of Theorem~\ref{Th:TCL} follows the lines of the proof of Theorem 25  of \cite{DelyonCours}. This theorem is an application of Theorem 24 of \cite{DelyonCours} in the case of Markovian dynamics. However, some assumptions required in Theorem 24 are not fulfilled  in our case: the A-stability of the algorithm and the boundedness in infinite norm of the solution of the Poisson equation.

Consider the stochastic approximation:
\begin{equation}
	\label{eq:algoReste}
	s_k=s_{k-1}+\gamma_k h(s_{k-1})+ \gamma_k \eta_k\,,
\end{equation}
where  the remainder term is decomposed as follows:
\begin{equation}
	\label{eq:eta}
	\eta_k = \xi_k + \nu_k -\nu_{k-1} + r_k
\end{equation}
with
\begin{eqnarray}
	\xi_k &=& g_{\hte(s_{k-1})}(\hid_{k}) -\ntrans_{\hte(s_{k-1})}g_{\hte(s_{k-1})}(\hid_{k-1})\nonumber 
	\\
	\label{eq:xi}\\
	\nu_k &=& - \ntrans_{\hte(s_k)} g_{\hte(s_k)} (\hid_k) 
	\label{eq:nu}\\
	r_k &=& \ntrans_{\hte(s_k)} g_{\hte(s_k)} (\hid_k) - \ntrans_{\hte(s_{k-1})} g_{\hte(s_{k-1})} (\hid_k)\nonumber\\
	\label{eq:r}
\end{eqnarray} 
and for any $s\in \Sr $, $g_{\hte(s)}$ is a solution of the Poisson equation $g-\ntrans_{\hte(s)} g = H_s - p_{\hte(s)} (H_s)$.

We recall Theorem 24 of \cite{DelyonCours}  with sufficient assumptions for our setting.

\begin{theorem}[Adapted from Theorem 24 \cite{DelyonCours}]
	\label{thm24bis}
Let assumptions \textbf{(N1)} and \textbf{(N3)} be fulfilled. Furthermore, assume that for some matrix $U$, some $\varepsilon >0$ and some positive random variables $X,X',X''$:
\begin{eqnarray}
	\label{eq:xiMartingale} \mbox{The sequence }  (\xi_i) \mbox{ is a } \mathcal{F}- \mbox{martingale} \\
	\label{eq:xiL2} \sup\limits_{i\in \N} \|\xi_i\|_{2+\varepsilon} <\infty\\
	\label{eq:rn} \lim\limits_{k\to\infty} \gamma_k^{-1/2} \|X r_k\|_1 =0\\
	\label{eq:nun}\lim\limits_{k\to\infty} \gamma_k^{1/2} \|X' \nu_k\|_1 =0\\
	\label{eq:varemp} \lim\limits_{k\to\infty} \gamma_k \|X'' \sum\limits_{i=1}^k (\xi_i\xi_i^T -U) \|_1 =0
\end{eqnarray} 
where $\mathcal{F}=(\mathcal{F}_i)_{i \in \mathbb{N}}$ is the increasing family of $\sigma-$algebra generated by the random variables $(s_0,\hid_1, ... , \hid_i)$. Then 
\begin{equation}
	\frac{s_k-s^*}{\sqrt{\gamma_k}} \to_\mathcal{L} \mathcal{N}(0,V)
\end{equation}
where $V$ is the solution of the following Lyapunov equation $U+JV +VJ^T =0$.
\end{theorem}

The result of Theorem 24 still holds replacing assumption  \textbf{(C)} of \cite{DelyonCours} by \textbf{(N1)}. Indeed, it is sufficient to establish that  the random variable $\gamma_k^{-1/2} \sum\limits_{i=0}^k \exp[(t_k-t_i)J]\gamma_i r_i$  converges toward $0$ in probability where $t_i = \sum\limits_{j=1}^i \gamma_j$. Theorem 19 and Proposition 39 of \cite{DelyonCours} can be applied in expectation. Theorems 23 and 20 of \cite{DelyonCours} also still hold with assumption \textbf{(N1)}.\\

We now prove that assumptions of Theorem \ref{thm24bis} hold. \\

By definition of $\xi_i$ it is obvious that \eqref{eq:xiMartingale} is fulfilled. Moreover, the following lemma proves that there exists   $\varepsilon >0 $ such that \eqref{eq:xiL2} holds with $X=1$.

\begin{lemma}
	\label{lemmaMartin}
	For all $\varepsilon > 0$, the sequence $\left(\xi_k\right)$ is in $L^{2+\varepsilon}$.
\end{lemma}
\begin{proof}
	We use the convexity of the function $x \mapsto x^{2+\varepsilon}$. Indeed, we have
\begin{multline*} 
|g_{\hte(s_{k-1})}(\hid_k)-\ntrans_{\hte(s_{k-1})}g_{\hte(s_{k-1})}(\hid_{k-1})|^{2+\varepsilon}  \\ \leq (|g_{\hte(s_{k-1})}(\hid_k)|+|\ntrans_{\hte(s_{k-1})}g_{\hte(s_{k-1})}(\hid_{k-1})|)^{2+\varepsilon}  \\
\leq
  C_\varepsilon(|g_{\hte(s_{k-1})}(\hid_k)|^{2+\varepsilon} + |\ntrans_{\hte(s_{k-1})}g_{\hte(s_{k-1})}(\hid_{k-1})|^{2+\varepsilon}) \,,
\end{multline*}
where $C_\varepsilon=\frac{1}{2^{3+\varepsilon}}$.
	
Applying the drift condition,  we get~:
\begin{multline*}
	\mathbb{E}(||\xi_k ||^{2+\varepsilon}  |  \mathcal{F}_{k-1})  \leq  C_\varepsilon \left( \mathbb{E}( |g_{\hte(s_{k-1})}(\hid_k)|^{2+\varepsilon}  \ | \ \mathcal{F}_{k-1})  \right.\\ 
\left.
	+ |\mathbb{E}(\ntrans_{\hte(s_{k-1})}g_{\hte(s_{k-1})}(\hid_{k-1})|\mathcal{F}_{k-1})|^{2+\varepsilon}
	\right) \\
	\leq  C \  \mathbb{E}( V(\hid_k)^{2+\varepsilon} + V (\hid_{k-1})^{2+\varepsilon}|\mathcal{F}_{k-1}))
 \\
 \leq 	C \left( \lambda V^{2+\varepsilon}(\hid_{k-1}) + 1 \right) \,.
\end{multline*} 
Finally taking the expectation after induction as in Lemma~\ref{lem:3} leads to:
\begin{equation*}
\mathbb{E} (||\xi_k^{2+\varepsilon}||)	\leq  C V^{2+\varepsilon}(\hid_{0}) < +\infty\,.
\end{equation*}

\end{proof}

Let us now focus on Equation \eqref{eq:rn}. Thanks to the H\"older property of our kernel and the fact that $H_{s_k}$ belongs to $\mathcal{L}_V$:

\begin{eqnarray*}
	\|r_k\|_1 &= &\mathbb{E}[|\ntrans_{\hte(s_k)} g_{\hte(s_k)} (\hid_k) - \ntrans_{\hte(s_{k-1})} g_{\hte(s_{k-1})} (\hid_k)|]\\
	&\leq & C \mathbb{E}[V^q(\hid_k) |s_k - s_{k-1}|^a]\\
	&\leq & C \mathbb{E}[V^{q+1}(\hid_k)] \gamma_k^a \\
	&\leq & C \gamma_k^a
\end{eqnarray*}
where the last inequality comes from the drift property.
Since the H\"older property is true for any $0<a<1$, we can choose $a>1/2$ which leads to the conclusion.\\

To prove Equation \eqref{eq:nun}, we note that using the drift condition as in the previous lemma, $\mathbb{E}(\|\nu_k\|)$ is uniformly bounded in $k$. Since the step-size sequence $(\gamma_k)_k$ tends to zero, the result follows with $X'=1$.\\

We follow the lines of the proof of Theorem 25 of \cite{DelyonCours} to establish Equation \eqref{eq:varemp}.
As in the proof of Lemma \ref{lemmaMartin}, we use the drift property coupled with our H\"older condition in $\mathcal{L}_V$-norm instead of the usual H\"older condition  considered by~\cite{DelyonCours} which is denoted (MS).\\

This concludes the proof of Theorem \ref{thm24bis}. \\

Applying Theorem \ref{thm24bis} allows to prove the first part of Theorem \ref{Th:TCL}. Assumption (\textbf{N2}) enables to characterize the covariance matrix $\Gamma$.
The Delta method gives the result on the sequence $(\theta_k)$ achieving the proof of our Central Limit Theorem.


\bibliographystyle{spmpsci}
\bibliography{bibcomp}


%
%



\end{document}